\documentclass[12pt]{amsart}
\usepackage[latin1]{inputenc}
\usepackage{indentfirst}
\usepackage{amsfonts}
\usepackage{graphicx}
\usepackage{amsmath}
\usepackage{amssymb}
\usepackage{latexsym}
\usepackage{amscd}
\usepackage{xcolor}
\usepackage{float}
\usepackage[pagebackref]{hyperref}
\usepackage[linesnumbered,lined,ruled,commentsnumbered]{algorithm2e}

\newcommand{\F}{\mathbb {F}}

\newtheorem{theorem}{Theorem}[section]
\newtheorem{definition}[theorem]{Definition}
\newtheorem{lemma}[theorem]{Lemma}

\newtheorem{proposition}[theorem]{Proposition}
\newtheorem{remark}[theorem]{Remark}

\begin{document}

\title[Pairs of $r$-primitive and $k$-normal elements]{Pairs of $r$-primitive and $k$-normal elements
in finite fields}

\author{Josimar J.R. Aguirre and Victor G.L. Neumann}

\maketitle


\vspace{8ex}
\noindent
\textbf{Keywords:} $r$-primitive element, $k$-normal element, normal basis, finite 
fields.\\
\noindent
\textbf{MSC:} 12E20, 11T23

\begin{abstract}
Let $\mathbb{F}_{q^n}$ be a finite field with $q^n$ elements and $r$ be a positive divisor of $q^n-1$. An element $\alpha \in \mathbb{F}_{q^n}^*$ is called \textit{$r$-primitive} if its multiplicative order is $(q^n-1)/r$. Also, $\alpha \in \mathbb{F}_{q^n}$ is \textit{$k$-normal} over $\mathbb{F}_q$ if the
greatest common divisor of the polynomials $g_{\alpha}(x) = \alpha x^{n-1}+ \alpha^q x^{n-2} +  \ldots + \alpha^{q^{n-2}}x + \alpha^{q^{n-1}}$ and $x^n-1$ in $\mathbb{F}_{q^n}[x]$ has degree $k$. These concepts generalize the ideas of primitive and normal elements, respectively. In this paper, we consider non-negative integers $m_1,m_2,k_1,k_2$, positive integers $r_1,r_2$ and rational functions $F(x)=F_1(x)/F_2(x) \in \mathbb{F}_{q^n}(x)$ with $\deg(F_i) \leq m_i$ for $i\in\{ 1,2\}$ satisfying certain conditions and we present
sufficient conditions for the existence of $r_1$-primitive $k_1$-normal elements $\alpha \in \mathbb{F}_{q^n}$ over $\mathbb{F}_q$, such that $F(\alpha)$ is an $r_2$-primitive $k_2$-normal element over $\mathbb{F}_q$. Finally as an example 
we study the case where $r_1=2$, $r_2=3$, $k_1=2$, $k_2=1$, $m_1=2$ and $m_2=1$,  with $n \ge 7$.
\end{abstract}

\section{Introduction}

For a positive integer $n$ and a prime power $q$, let $\mathbb{F}_{q^n}$ be the finite field with $q^n$ elements. 
We recall that the multiplicative group $\mathbb{F}_{q^n}^*$ is cyclic, and an element $\mathbb{F}_{q^n}$ 
is called \textit{primitive} if its multiplicative order is $q^n-1$. Primitive elements have many applications
in 
cryptography (see \cite{blum}, \cite{mel}). Let $r$ be a positive divisor of 
$q^n-1$. An element $\alpha \in \mathbb{F}_{q^n}^*$ is called \textit{$r$-primitive} if its multiplicative order 
is $(q^n-1)/r$, so primitive elements in the usual sense are 1-primitive elements. In \cite{Cohen2021}, \cite{Cohen2022} the authors found a characteristic function for the $r$-primitive elements. These elements that have high multiplicative order, 
without necessarily being primitive, are of great practical interest because 
they may replace primitive elements in several
applications.

Also, $\alpha \in \mathbb{F}_{q^n}$ is normal over $\mathbb{F}_q$ if the set 
$B_{\alpha}=\{\alpha, \alpha^q, \ldots, \alpha^{q^{n-1}} \}$ spans
$\mathbb{F}_{q^n}$ as an $\mathbb{F}_q$-vector space. In this case $B_{\alpha}$ is called a normal basis. 
Normal bases have many applications in the computational theory due to their efficient implementation in finite field arithmetic \cite{gao}. We will use an equivalence to define $k$-normal elements (see \cite[Theorem 3.2]{knormal}).
An element $\alpha \in \mathbb{F}_{q^n}$ is said to be $k$-normal over $\mathbb{F}_q$ if $\alpha$ gives rise to a 
basis $\{\alpha,\alpha^{q},\ldots,\alpha^{q^{n-k-1}}\}$ of a $q$-modulus of dimension $n-k$ over $\mathbb{F}_q$.
This definition implies that normal elements in the usual sense are $0$-normal. The $k$-normal elements can be used to reduce the multiplication process in finite fields (see \cite{negre}).

There are several criteria in the literature for the existence of $k$-normal elements (see
e.g. \cite{sozaya}, \cite{lucas}, \cite{zhang}). In \cite{lucas1} the authors worked out the case $k=1$, and established the Primitive 1-Normal Theorem. In \cite{AN}, the authors showed conditions for the existence of primitive $2$-normal elements. Generalizing these ideas, in \cite{AN2}, \cite{RSS}  the authors presented conditions for the existence of $r$-primitive
$k$-normal elements in $\mathbb{F}_{q^n}$ over $\mathbb{F}_q$.

The study of pairs of elements with certain properties has begun with the work of Cohen and Huczynska. 
They showed the Strong Normal Basis Theorem (see \cite{CH2}) which states
that, except for a few pairs $(q,n)$, one can find an element $\alpha \in \mathbb{F}_{q^n}$ such that $\alpha$ 
and $\alpha^{-1}$ are primitive and normal over $\mathbb{F}_{q}$.  Later, Kapetanakis (see \cite{kap}) proved that there exists an element $\alpha \in \mathbb{F}_{q^n}$ such that $\alpha$ and $(a\alpha+b)/(c\alpha+d)$, 
with $a, b, c, d \in \mathbb{F}_q$, are primitive and normal over $\mathbb{F}_q$, except for a
few combinations of $q,n$ and $a, b, c, d$. A few years ago, many results in this sense were obtained for particular cases of the characteristic of the field, as well as on the degrees of a rational function, i.e, conditions for which $(\alpha,F(\alpha))$ is a pair formed by primitive and normal elements (see \cite{anju}, \cite{haz}, \cite{haz2}, \cite{cgnt2}). In \cite{rania}, the authors showed conditions for the existence of an element $\alpha \in \mathbb{F}_{q^n}$ such that $\alpha$ and $\alpha^{-1}$ are $r$-primitive $k$-normal elements over $\mathbb{F}_q$.

In this paper, we study conditions for the existence of pairs $(\alpha, F(\alpha))$ with $\alpha \in \mathbb{F}_{q^n}$ and 
$F(x) \in \mathbb{F}_{q^n}(x)$ satisfying certain conditions, such that $\alpha$ is an $r_1$-primitive $k_1$-normal element over $\mathbb{F}_q$ and $F(\alpha)$ is an $r_2$-primitive $k_2$-normal element over $\mathbb{F}_{q}$, where $r_1,r_2$ are positive divisors of $q^n-1$ and $k_1,k_2$ are the degrees of some polynomials over $\mathbb{F}_q$ that divide $x^n-1$. This paper is organized as follows.
In Section 2, we provide background material that is used along
the paper. In Section 3, we present the general condition for the existence of these pairs of $r$-primitive
$k$-normal elements in $\mathbb{F}_{q^n}$ over $\mathbb{F}_q$, as well as another derived condition using sieve method. 
In Section 4, we provide some numerical results.

\section{Preliminaries}

In this section, we present some definitions and results required in
this paper.
We refer the reader to \cite{LN} 
for
basic results on finite fields.

\begin{definition}\label{def-upsilon}
For $m_1,m_2 \in \mathbb{N}$, define $\Upsilon_q(m_1,m_2)$ as the set of rational functions 
$\frac{F_1}{F_2} \in \mathbb{F}_q(x)$ such that:
\begin{enumerate}
\item $\deg(F_1) \leq m_1, \deg(F_2) \leq m_2$;
\item $\gcd(F_1,F_2)=1$;
\item there exist $m \in \mathbb{N}$ and an irreducible monic polynomial $g \in \mathbb{F}_q[x] \setminus \{x\}$ 
such that $\gcd(m,q-1)=1$, $g^m \mid F_1F_2$ and $g^{m+1} \nmid F_1F_2$.  
\end{enumerate}
\end{definition}

For a positive integer $m$, $\phi(m)$ denotes the Euler totient function, $\mu(m)$ denotes
the M\"obius function, $\pi(m)$ denotes the number of primes less than or equal to $m$ and
$\textrm{rad}(m)$ denotes the greatest square-free divisor of $m$.
\begin{definition}
\begin{enumerate}
\item[(a)] Let $f(x)\in \mathbb{F}_{q}[x]$. The Euler totient function for polynomials over $\mathbb{F}_q$ is given by
$$
\Phi_q(f)= \left| \left( \dfrac{\mathbb{F}_q[x]}{\langle f \rangle} \right)^{*} \right|,
$$
where $\langle f \rangle$ is the ideal generated by $f(x)$ in $\mathbb{F}_q[x]$.
\item[(b)] If $t$ is a positive integer (or a monic polynomial over $\mathbb{F}_q$), $W(t)$ denotes the number of square-free (monic) divisors of $t$.
\item[(c)] If $f(x)\in \mathbb{F}_{q}[x]$ is a monic polynomial, the polynomial M\"obius function $\mu_q$ is given by $\mu_q(f)=0$ if $f$ is not square-free and $\mu_q(f)=(-1)^r$ if $f$ is a product of $r$ distinct monic irreducible factors over $\mathbb{F}_q$.
\end{enumerate}
\end{definition}

\subsection{Freeness and characters.}
We present the concept of \textit{freeness}, introduced by Carlitz \cite{carlitz} and Davenport \cite{davenport1},
and refined by Lenstra and Schoof (see \cite{lenstra}). This 
concept is useful in the construction of certain characteristic functions over finite fields.

The additive group $\mathbb{F}_{q^n}$ is an $\mathbb{F}_q[x]$-module where the action
is given by $f \circ \alpha =\displaystyle  \sum_{i=0}^r a_i \alpha^{q^i}$
for $f=\displaystyle \sum_{i=0}^r a_ix^i\in \mathbb{F}_q[x]$.
An element $\alpha \in \F_{q^n}$ has $\F_q$-order $h\in \mathbb{F}_q[x]$
if $h$ is the lowest degree monic polynomial such that $h \circ \alpha=0$.
The $\mathbb{F}_q$-order of $\alpha$ will be denoted by $\mathrm{Ord} (\alpha)$.
It is known that the $\F_q$-order of an element $\alpha \in \F_{q^n}$
divides $x^n-1$. 
An additive character $\chi$ 
of $\F_{q^n}$ is a
group homomorphism of $\mathbb{F}_{q^n}$ to $\mathbb{C}^*$.
The
group of additive characters $\widehat{\mathbb{F}}_{q^n}$ becomes 
an $\mathbb{F}_q[x]$-module by defining
$f\circ \chi (\alpha)=\chi(f \circ \alpha)$ for $\chi \in \widehat{\mathbb{F}}_{q^n}$.
An additive character $\chi$ has $\F_q$-order $h \in \mathbb{F}_q[x]$
if $h$ is the monic polynomial of smallest degree such
that $h \circ \chi$
is the trivial additive character.
The $\mathbb{F}_q$-order of $\chi$ will be denoted by $\mathrm{Ord} (\chi)$.

Let $g\in \mathbb{F}_q[x]$ be a  divisor of $x^n-1$. We say that an element
$\alpha \in \mathbb{F}_{q^n}$ is $g$-free
if for every polynomial $h \in \mathbb{F}_q[x]$  such that $h \mid g$ and $h\neq 1$, there is no element
$\beta \in \F_{q^n}$ satisfying $\alpha= h \circ \beta$.
From \cite[Theorem 13.4.4.]{galois}, we have that
for any $\alpha \in \F_{q^n}$ we get
$$
\Omega_g(\alpha)=
\Theta(g) \int_{h|g} \chi_h(\alpha)
=\left\{
\begin{array}{ll}
	1 \quad & \text{if } \alpha \text{ is } g\text{-free}, \\
	0            & \text{otherwise,}
\end{array}
\right.
$$
where 
$\Theta(g)= \frac{\Phi_q(g)}{q^{\deg(g)}}$,
$\displaystyle\int_{h|g} \chi_h$ denotes the sum
$\displaystyle \sum_{h|g} \frac{\mu_q(h)}{\Phi_q(h)} \sum_{(h)} \chi_h$,
$\displaystyle \sum_{h|g}$ runs over all the monic divisors $h\in \mathbb{F}_q[x]$ of $g$,
$\chi_h$ is an additive character of $\F_{q^n}$ and the sum
$\displaystyle \sum_{(h)} \chi_h$ runs over all additive characters of $\mathbb{F}_q$-order $h$.
It is known that there exist $\Phi_q(h)$ of those characters.
It is well known that an element $\alpha \in \mathbb{F}_{q^n}$ is normal if and only if $\alpha$ is $(x^n-1)$-free.

A multiplicative character $\eta$ 
of $\F_{q^n}^*$ is a
group homomorphism of $\F_{q^n}^*$ to $\mathbb{C}^*$.
The group of multiplicative characters $\widehat{\mathbb{F}}_{q^n}^*$ becomes 
a $\mathbb{Z}$-module by defining
$\eta^r(\alpha)=\eta(\alpha^r)$ for $\eta \in \widehat{\mathbb{F}}_{q^n}^*$,
$\alpha \in \mathbb{F}_{q^n}^*$ and $r \in \mathbb{Z}$.
The order of a multiplicative character $\eta$ is the least positive integer $d$ such that
$\eta (\alpha)^d=1$. 


There are some works which characterize $r$-primitive elements of $\mathbb{F}_{q^n}$ using characters,
like \cite{Cohen2021} and \cite{Cohen2022}. We will follow \cite{Cohen2022} and, as in that work, for positive integers $a$ and $b$, we set $a_{(b)}=\frac{a}{\gcd(a,b)}$.

\begin{definition}{\cite[Definition 3.1]{Cohen2022}}
	For a divisor $r$ of $q^n-1$
	and a divisor $R$ of $\frac{q^n-1}{r}$, let $\mathcal{C}_r$ be the cyclic multiplicative subgroup of $\mathbb{F}_{q^n}^*$
	of order $\frac{q^n-1}{r}$. We say that an element $\alpha \in  \mathbb{F}_{q^n}^*$ is $(R,r)$-free if
	$\alpha \in \mathcal{C}_r$ and $\alpha$ is $R$-free in $\mathcal{C}_r$, i.e., 
	if $\alpha=\beta^s$ with $\beta \in \mathcal{C}_r$ and $s \mid R$, then $s=1$.
\end{definition}

\begin{remark}
	An element $\alpha \in  \mathbb{F}_{q^n}^*$ is $r$-primitive if and only if $\alpha$ is $(\frac{q^n-1}{r},r)$-free.
\end{remark}

Let $\mathbb{I}_{R,r}$ be the characteristic function of $(R,r)$-free elements of $\mathbb{F}_{q^n}^*$, i.e.
$$
\mathbb{I}_{R,r}(\alpha) =
\left\{
\begin{array}{ll}
	1, \quad & \text{if } \alpha \text{ is } (R,r)\text{-free}, \\
	0,            & \text{otherwise.}
\end{array}
\right.
$$

From \cite[Proposition 3.6]{Cohen2022}, 
for any $\alpha \in \mathbb{F}_{q^n}^*$ we get
$$\mathbb{I}_{R,r}(\alpha) = 
\frac{\theta(R)}{r} \mathop{\int}_{d_{(r)}|Rr} \eta_d(\alpha),
$$
where $\theta(R)=\frac{\phi(R)}{R}$,
$\displaystyle\mathop{\int}_{d_{(r)}|Rr} \eta_d$ stands for the sum
$\displaystyle \sum_{d|Rr} \frac{\mu(d_{(r)})}{\phi(d_{(r)})} \sum_{(d)} \eta_d$,
$\eta_d$ is a multiplicative character of $\F_{q^n}^*$  and the sum
$\displaystyle \sum_{(d)} \eta_d$ runs over all the multiplicative characters of order $d$.

To bound the sums above, we will use the following result.

\begin{lemma}{\cite[Lemma 2.5]{Cohen2022}}\label{lemmacohen2022}
	For any positive integers $R$, $r$, we have that
	$$
	\sum_{d\mid R} \frac{|\mu(d_{(r)})|}{\phi(d_{(r)})} \cdot \phi(d) = \gcd(R,r)\cdot W(\gcd(R,R_{(r)})).
	$$
\end{lemma}

\begin{remark}\label{construct-k-r}
L. Reis has given a method to construct $k$-normal elements. Let $\beta \in \mathbb{F}_{q^n}$ be a normal element and $f\in \mathbb{F}_q[x]$ be a divisor of $x^n-1$ of degree $k$, then $\alpha = f \circ \beta$ is $k$-normal
(see \cite[Lemma 3.1]{lucas}). In the same way, if $\beta \in \mathbb{F}_{q^n}$ is a primitive element, then $\beta^r$ is $r$-primitive
for any divisor $r$ of $q^n-1$.
\end{remark}

We also have that $\mathbb{F}_{q^n}^*$ and $\widehat{\mathbb{F}}_{q^n}^*$
are isomorphic as $\mathbb{Z}$-modules, and
$\mathbb{F}_{q^n}$ and $\widehat{\mathbb{F}}_{q^n}$
are isomorphic as $\mathbb{F}_q[x]$-modules (see \cite[Theorem 13.4.1]{galois}).



\begin{definition}\label{char0}
For any $\alpha \in \mathbb{F}_{q^n}$, we define the following character sum:
$$
I_0(\alpha) = \frac{1}{q^n} \sum_{\psi \in \widehat{\mathbb{F}}_{q^n}} \psi(\alpha).
$$
\end{definition}
Note that $I_0(\alpha)=1$ if $\alpha=0$, and $I_0(\alpha)=0$ otherwise by the character orthogonality property.

\subsection{Estimates.} 
To finish this section, we present some estimates that are used along the next sections.



\begin{lemma}\label{case-b}\cite[Lemma 2.5]{AN2} 
Let $f \in \mathbb{F}_q[x]$ be a divisor of $x^n-1$ of degree $k$ and
let $\chi$ and $\psi$ be additive characters. Then
$$
\sum_{\beta \in \mathbb{F}_{q^n}}
\chi(\beta)\psi(f \circ \beta)^{-1} =
\left\{
\begin{array}{ll}
	q^n & \text{if } \chi = f \circ \psi, \\
	0 & \text{if } \chi \neq f \circ \psi .
\end{array}
\right.
$$
Furthermore, for a given additive character $\chi$, 
the set
$\hat{f}^{-1}(\chi)=\{\psi \in  \widehat{\mathbb{F}}_{q^n} \mid \chi = f \circ \psi\}$
has $q^k$ elements if 
$\mathrm{Ord}(\chi) \mid \frac{x^n-1}{f}$, and it is the empty set
if $\mathrm{Ord}(\chi) \nmid \frac{x^n-1}{f}$.
\end{lemma}


%
%

The next result is a combination of \cite[Theorem 5.5]{Fu} and a special case of \cite[Theorem 5.6]{Fu}.

\begin{lemma}\label{cotaparaf}
Let $v(x),u(x) \in \mathbb{F}_{q^n}(x)$ be rational functions. Write $v(x)=\prod_{j=1}^k s_j(x)^{n_j}$, 
where $s_j(x) \in \mathbb{F}_{q^n}[x]$ are irreducible polynomials, pairwise non-associated, and $n_j$ 
are non-zero integers. Let $D_1=\sum_{j=1}^k \deg(s_j)$, $D_2=\max\{\deg(u),0\}$, $D_3$ be the degree of the 
denominator of $u(x)$ and $D_4$ be the sum of degrees of those irreducible polynomials dividing the denominator 
of $u$, but distinct from $s_j(x)$ ($j=1,\ldots,k$). Let $\eta$ and $\psi$ be, respectively, a multiplicative 
character and a non-trivial additive character of $\mathbb{F}_{q^n}$.
\begin{enumerate}
\item[a)] Assume that $v(x)$ is not of the form $r(x)^{ord(\eta)}$ in $\mathbb{F}(x)$, where $\mathbb{F}$ is the 
algebraic closure of $\mathbb{F}_{q^n}$. Then
$$
\Big| \sum_{\substack{\alpha \in \mathbb{F}_{q^n} \\ v(\alpha) \neq 0, v(\alpha) \neq \infty}} 
\eta(v(\alpha)) \Big|  \leq (D_1-1)q^{\frac{n}{2}}.
$$
\item[b)] Assume that $u(x)$ is not of the form $r(x)^{q^n}-r(x)$ in $\mathbb{F}(x)$, where $\mathbb{F}$ is the 
algebraic closure of $\mathbb{F}_{q^n}$. Then
$$
\Big| \sum_{\substack{\alpha \in \mathbb{F}_{q^n} \\ v(\alpha) \neq 0, v(\alpha) \neq \infty \\ u(\alpha) \neq \infty}} 
\eta(v(\alpha)) \psi(u(\alpha)) \Big|  \leq (D_1+D_2+D_3+D_4-1)q^{\frac{n}{2}}.
$$
\end{enumerate}
\end{lemma}

\section{General results}\label{sectiongen}

Let $r_1$, $r_2$ be positive divisors of $q^n-1$ and
let $f_1,f_2 \in \mathbb{F}_q[x]$ be monic factors of $x^n-1$ of degrees
$k_1$, $k_2$, respectively. Let also $m_1$, $m_2$ be non-negative integers
such that  $1 \le m_1+m_2 < q^{n/2}$ and let $F =\frac{F_1}{F_2} \in \Upsilon_{q^n}(m_1,m_2)$.
Also, let $R_1$, $R_2$ be divisors of $\frac{q^n-1}{r_1}$
and $\frac{q^n-1}{r_2}$, respectively, and $g_1,g_2 \in \mathbb{F}_q[x]$ 
be monic divisors of $x^n-1$.

We want to determine conditions on $q$ and $n$ for which
there exists an element $\alpha \in \mathbb{F}_{q^n}$ $r_1$-primitive
$k_1$-normal over $\mathbb{F}_q$ such that $F(\alpha)\in \mathbb{F}_{q^n}$ is $r_2$-primitive
$k_2$-normal over $\mathbb{F}_q$. For this, the following definition plays an important role.


\begin{definition}\label{def-NrfmT}
We denote by $N_F(R_1,R_2,g_1,g_2)$ (when $g_1=g_2=:g$ we will write $N_F(R_1,R_2,g)$) the sum
$$
\sum_{{
			\substack{\alpha \in \mathbb{F}_{q^n}^*\backslash  S_F \\
				\beta_1,\beta_2 \in \mathbb{F}_{q^n} }
	}} 
	\Big( 
	\mathbb{I}_{R_1,r_1}(\alpha) \mathbb{I}_{R_2,r_2}(F(\alpha)) \Omega_{g_1}(\beta_1) \Omega_{g_2}(\beta_2)
	I_0\left(\alpha - f_1 \circ \beta_1  \right) 
	I_0\left(F(\alpha) - f_2 \circ \beta_2 \right)  \Big),
$$
where $S_F:=\{\alpha \in \mathbb{F}_{q^n} \ \mid \  F_1(\alpha)=0 \textrm{ or } F_2(\alpha)=0\}$.
\end{definition}

From the definitions of $\mathbb{I}_{R_i,r_i}$, $\Omega_{g_i}$ ($i\in\{ 1,2\}$), $I_0$ and Definition \ref{def-NrfmT}, 
$N_F(R_1,R_2,g_1,g_2)$ counts 
the numbers of triples $(\alpha, \beta_1, \beta_2) \in (\mathbb{F}_{q^n}^*\backslash  S_F) \times (\mathbb{F}_{q^n})^2$ 
such that $\alpha$ is $(R_1,r_1)$-free, $F(\alpha)$ is $(R_2,r_2)$-free,
$\beta_1$ is $g_1$-free, $\beta_2$ is $g_2$-free,
$\alpha=f_1 \circ \beta_1$ and
$F(\alpha)=f_2 \circ \beta_2$. 
In particular, if $N_F(\frac{q^n-1}{r_1},\frac{q^n-1}{r_2},x^n-1)>0$, then there exists a triple 
$(\alpha, \beta_1, \beta_2) \in (\mathbb{F}_{q^n}^*\backslash  S_F) \times (\mathbb{F}_{q^n})^2$ 
such that $\alpha=f_1 \circ \beta_1$ is an $r_1$-primitive $k_1$-normal element of 
$\mathbb{F}_{q^n}$ over $\mathbb{F}_q$ and 
$F(\alpha)= f_2 \circ \beta_2$ is an $r_2$-primitive $k_2$-normal element of $\mathbb{F}_{q^n}$ over $\mathbb{F}_q$.

We need to find lower estimates for the sum above, in order to guarantee the positivity of $N_F(\frac{q^n-1}{r_1},\frac{q^n-1}{r_2},x^n-1)$. We have the following result.

\begin{theorem}\label{principal}
Let $M=\max\{2(m_1+m_2),m_1+3m_2+1\}$.
If 
$$
q^{\frac{n}{2}-k_1-k_2}\geq M r_1r_2 W(R_1)W(R_2)W(\gcd(g_1,\frac{x^n-1}{f_1}))W(\gcd(g_2,\frac{x^n-1}{f_2})),
$$ 
then $N_F(R_1,R_2,g_1,g_2)>0$. 

In particular, if 
$q^{\frac{n}{2}-k_1-k_2}\geq M r_1r_2 W(\frac{q^n-1}{r_1})W(\frac{q^n-1}{r_2})
W(\frac{x^n-1}{f_1})
W(\frac{x^n-1}{f_2})$,
then there exists an $r_1$-primitive $k_1$-normal element $\alpha \in \mathbb{F}_{q^n}$ over $\mathbb{F}_q$ such that 
$F(\alpha) \in \mathbb{F}_{q^n}$ is an $r_2$-primitive $k_2$-normal element over $\mathbb{F}_q$.
\end{theorem}
\begin{proof}
Let
$S_F$ as in Definition \ref{def-NrfmT}.
From definitions of $\mathbb{I}_{R,r}$, $\Omega_{g}$, $I_0$ and Definition \ref{def-NrfmT}, we have that $N_F(R_1,R_2,g_1,g_2)$ is equal to
$$
\frac{\theta(R_1)\theta(R_2)\Theta(g_1)\Theta(g_2)}{r_1r_2 }
\mathop{\int}_{
\substack{
	{d_1}_{(r_1)}|R_1r_1\\
	{d_2}_{(r_2)}|R_2r_2 }  }
\mathop{\int}_{
\substack{h_1|g_1 \\ h_2|g_2}	}
\sum_{\psi_1,\psi_2 \in \widehat{\mathbb{F}}_{q^n}}
\tilde{S}(\eta_{d_1},\eta_{d_2},\chi_{h_1},\chi_{h_2},\psi_1,\psi_2),
$$
where
\begin{align*}
	\tilde{S}(\eta_{d_1},\eta_{d_2},\chi_{h_1},\chi_{h_2},\psi_1,\psi_2) & =  
\frac{1}{q^{2n}}
\sum_{\alpha \in \mathbb{F}_{q^n}^* \setminus S_F}
	        \eta_{d_1}(\alpha)\eta_{d_2}(F(\alpha)) \psi_1(\alpha) \psi_2(F(\alpha))  
	\times\\
	& \times \sum_{\beta_1 \in \mathbb{F}_{q^n}} \chi_{h_1}(\beta_1) \psi_1^{-1}(f_1 \circ \beta_1) 
	\sum_{\beta_2 \in \mathbb{F}_{q^n}} \chi_{h_2}(\beta_2) \psi_2^{-1}(f_2 \circ \beta_2).
\end{align*}
First note that from Lemma \ref{case-b}, 
if $\psi_ i \in \hat{f}_i^{-1}(\chi_{h_i})$, for $i\in\{ 1,2\}$, then we have
$$
\sum_{\beta_i \in \mathbb{F}_{q^n}}  \chi_{h_i}(\beta_i) \psi_i^{-1}(f_i \circ \beta_1)
= q^{n}.
$$
This sum is $0$ for $\psi_ i \notin \hat{f}_i^{-1}(\chi_{h_i})$ 
and  the set $\hat{f}_i^{-1}(\chi_{h_i})$ is empty if  $\mathrm{Ord}(\chi_{h_i})=h_i \nmid \frac{x^n-1}{f_i}$.
Defining  $\widetilde{g_i}=\gcd (g_i, \frac{x^n-1}{f_i})$,
we get that $N_F(R_1,R_2,g_1,g_2)$ equals
$$
\frac{\theta(R_1)\theta(R_2)\Theta(g_1)\Theta(g_2)}{r_1r_2 }
\mathop{\int}_{
	\substack{
		{d_1}_{(r_1)}|R_1r_1\\
		{d_2}_{(r_2)}|R_2r_2 }  }
\mathop{\int}_{
	\substack{h_1|\widetilde{g_1} \\ h_2| \widetilde{g_2}}	}
\sum_{
\substack{
\psi_ 1 \in \hat{f}_1^{-1}(\chi_{h_1})  \\
\psi_ 2 \in \hat{f}_2^{-1}(\chi_{h_2})
}
}
S(\eta_{d_1},\eta_{d_2},\psi_1,\psi_2),
$$
where
$$
S(\eta_{d_1},\eta_{d_2},\psi_1,\psi_2) =  
	\sum_{\alpha \in \mathbb{F}_{q^n}^* \setminus S_F}
\eta_{d_1}(\alpha)\eta_{d_2}(F(\alpha)) \psi_1(\alpha) \psi_2(F(\alpha)) .
$$
To find a lower bound for $N_F(R_1,R_2,g_1,g_2)$, we will bound 
$|S(\eta_{d_1},\eta_{d_2},\psi_1,\psi_2)|$.

Now we consider six cases.
\begin{enumerate}
\item[(i)] We first consider the case where $\eta_{d_1}$ and $\eta_{d_2}$ are trivial multiplicative characters, 
$\psi_1$ and $\psi_2$ are trivial additive characters, so that
$$
S(\eta_{d_1},\eta_{d_2},\psi_1,\psi_2) = |\mathbb{F}_{q^n}^* \backslash S_F|
\geq q^n - (m_1+m_2+1).
$$
\item[(ii)] Consider now the case where $\eta_{d_2}$ is a trivial multiplicative character,
$\psi_2$ is a trivial additive character, $\eta_{d_1}$ is any multiplicative character
of order $d_1$, and
$\psi_1$ is not a trivial additive character. From Lemma \ref{cotaparaf}(b) we get
\begin{eqnarray*}
|S(\eta_{d_1},\eta_{d_2},\psi_1,\psi_2)| & =&  
\Big|\sum_{\alpha \in \mathbb{F}_{q^n}^* \setminus S_F}
\eta_{d_1} (\alpha) \psi_1(\alpha)\Big| \\
& \leq&  
\Big|\sum_{\alpha \in \mathbb{F}_{q^n}^*}
\eta_{d_1} (\alpha) \psi_1(\alpha)\Big| + \Big|\sum_{\alpha \in  S_F}\eta_{d_1} (\alpha)  \psi_1(\alpha)\Big| \\
&\leq& q^{n/2} + m_1+m_2.
\end{eqnarray*}
\item[(iii)] If only $\eta_{d_1}$ is not a trivial multiplicative character, then
\begin{eqnarray*}
|S(\eta_{d_1},\eta_{d_2},\psi_1,\psi_2)| & =&  
\Big|\sum_{\alpha \in \mathbb{F}_{q^n}^* \setminus S_F} \eta_{d_1}(\alpha)\Big| 
	=
	\Big|\sum_{\alpha \in  S_F } \eta_{d_1}(\alpha)\Big| \\
	&\leq& m_1+m_2 ,
\end{eqnarray*}
since $\sum_{\alpha \in  \mathbb{F}_{q^n}^*} \eta_{d_1}(\alpha)=0$.
\end{enumerate}
Before to treat the cases where $\eta_{d_2}$ is not a  trivial multiplicative character or 
$\psi_2$ is not a trivial additive character, we will rewrite the expression
$S(\eta_{d_1},\eta_{d_2},\psi_1,\psi_2)$.

It is well-known (see e.g. \cite[Theorem 5.8.]{LN}) that there exists a multiplicative character
$\eta$ of order $q^n-1$ and integers $t_1,t_2 \in \{0,1,\ldots, q^n-2 \}$ such that
$\eta_{d_1} (\alpha)=\eta (\alpha^{t_1})$ and $\eta_{d_2} (\alpha)=\eta (\alpha^{t_2})$
for all $\alpha \in \mathbb{F}_{q^n}^*$,
and $t_i=0$ if and only if $\eta_{d_i}$ is the trivial multiplicative character, for $i\in\{ 1,2\}$.
Hence,
$\eta_{d_1}(\alpha)\eta_{d_2}(F(\alpha)) = \eta (\alpha^{t_1} F(\alpha)^{t_2})$.

Analogously, it is also known (see e.g. \cite[Theorem 5.7.]{LN}) that
for given additive characters $\psi_1$, $\psi_2$,
there exist elements $y_1,y_2 \in \mathbb{F}_{q^n}$ such that 
$\psi_1(\alpha)=\chi(y_1 \alpha)$ and $\psi_2(\alpha)=\chi(y_2\alpha)$ for all $\alpha \in \mathbb{F}_{q^n}$,
where $\chi$ is the canonical additive character of $\mathbb{F}_{q^n}$,
and $y_i=0$ if and only if $\psi_{i}$ is the trivial additive character for $i\in\{ 1,2\}$.
Hence
$\psi_1(\alpha)\psi_2(F(\alpha)) = \chi(y_1 \alpha+ y_2F(\alpha))$.
\begin{enumerate}
\item[(iv)] Suppose now that $\eta_{d_1}$ 
is any multiplicative character of order $d_1$, $\eta_{d_2}$ is not a trivial multiplicative character, 
$\psi_{1}$ and $\psi_2$ are trivial additive characters. In this case,
$$
S(\eta_{d_1},\eta_{d_2},\psi_1,\psi_2) =
	\sum_{\alpha \in \mathbb{F}_{q^n}^* \setminus S_F} \eta (\alpha^{t_1} F(\alpha)^{t_2}),
$$
with $t_2 \neq 0$.

From the proof of \cite[Theorem 3.2]{cgnt2}, we know that $v(x):=x^{t_1} F(x)^{t_2}$ is not of
the form $r(x)^{q^n-1}$ in $\mathbb{F}(x)$ where $\mathbb{F}$ is the algebraic closure of $\mathbb{F}_{q^n}$, 
thus we can use Lemma \ref{cotaparaf} (a).


Let $S_v$ be the set of elements $\alpha \in \mathbb{F}_{q^n}$ such that $v(\alpha)=0$ or $v(\alpha)$ is not defined.
If $0 \in S_v$, then $\mathbb{F}_{q^n}^* \backslash S_F = \mathbb{F}_{q^n} \backslash S_v$ and,
from Lemma \ref{cotaparaf}(a),
we have
$$
S(\eta_{d_1},\eta_{d_2},\psi_1,\psi_2) =
 \sum_{\alpha \in \mathbb{F}_{q^n}^* \setminus S_F} \eta (v(\alpha)) =
\sum_{\alpha \in \mathbb{F}_{q^n} \setminus S_v} \eta (v(\alpha)) ,
$$
and hence $|S(\eta_{d_1},\eta_{d_2},\psi_1,\psi_2)|  \leq (m_1+m_2)q^{n/2}$.

If $0 \notin S_v$, then
$$
S(\eta_{d_1},\eta_{d_2},\psi_1,\psi_2) =
 \sum_{\alpha \in \mathbb{F}_{q^n}^* \setminus S_F} \eta (v(\alpha)) =
\sum_{\alpha \in \mathbb{F}_{q^n} \setminus S_v} \eta (v(\alpha)) - \eta(v(0)),
$$
so that $|S(\eta_{d_1},\eta_{d_2},\psi_1,\psi_2)| \leq (m_1+m_2-1)q^{n/2}+ 1 \leq (m_1+m_2)q^{n/2}$.

\item[(v)] Now we assume that  $\eta_{d_1}$ 
is any multiplicative character of order $d_1$, $\eta_{d_2}$ is not a trivial multiplicative character,
$\psi_1$ is not a trivial additive character and $\psi_2$ is a trivial additive character.
Define $S_v$ as in the previous case.
If $0 \in S_v$, then $\mathbb{F}_{q^n}^* \backslash S_F = \mathbb{F}_{q^n} \backslash S_v$ and,
from Lemma \ref{cotaparaf}(b),
$$
S(\eta_{d_1},\eta_{d_2},\psi_1,\psi_2) =
\sum_{\alpha \in \mathbb{F}_{q^n}^* \setminus S_F} \eta (v(\alpha))\psi_1(\alpha)=
\sum_{\alpha \in \mathbb{F}_{q^n} \setminus S_v} \eta (v(\alpha))\psi_1(\alpha) ,
$$
and thus $|S(\eta_{d_1},\eta_{d_2},\psi_1,\psi_2)|  \leq (m_1+m_2+1)q^{n/2}$.

If $0 \notin S_v$, then
\begin{eqnarray*}
S(\eta_{d_1},\eta_{d_2},\psi_1,\psi_2) &=& 
\sum_{\alpha \in \mathbb{F}_{q^n}^* \setminus S_F} \eta (v(\alpha))\psi_1(\alpha)\\
&=&
 \sum_{\alpha \in \mathbb{F}_{q^n} \setminus S_v} \eta (v(\alpha))\psi_1(\alpha) - \eta(v(0)),
\end{eqnarray*}
so that $|S(\eta_{d_1},\eta_{d_2},\psi_1,\psi_2)| \leq (m_1+m_2)q^{n/2}+ 1 < (m_1+m_2+1)q^{n/2}$.

\item[(vi)] Finally, we consider the case where  $\eta_{d_1}$ 
is any multiplicative character of order $d_1$, $\eta_{d_2}$ is not a trivial multiplicative character and
$\psi_1$, $\psi_2$ are not trivial additive characters. 
In this case,
$$
S(\eta_{d_1},\eta_{d_2},\psi_1,\psi_2) =
\sum_{\alpha \in \mathbb{F}_{q^n}^* \setminus S_F} \eta (\alpha^{t_1} F(\alpha)^{t_2})\chi(y_1 \alpha + y_2 F(\alpha)),
$$
where $y_1,y_2 \neq 0$. Since $m_1+m_2 < q^{n/2}$ and $y_2 \neq 0$, the function $y_1 x + y_2 F(x)$
cannot be of the form $r(x)^{q^n-1}-r(x)$ for any $r(x) \in \mathbb{F}(x)$. Thus we may apply
Lemma \ref{cotaparaf}(b). As in the cases (iv) and (v), we may obtain different
inequalities, but in both cases we get
$$
|S(\eta_{d_1},\eta_{d_2},\psi_1,\psi_2)| \leq M q^{n/2},
$$
where $M= m_1+m_2 +1+ \max(\deg(y_1 x + y_2 F(x)),0)+m_2+m_2-1$.

We have that
$\max(\deg(y_1 x + y_2 F(x)),0) \leq \max(m_2+1,m_1) - m_2$,
since $y_1 x + y_2 F(x)= \frac{y_1xF_2(x)+y_2 F_1(x)}{F_2(x)}$. Hence we may assume that
$M=\max \{2(m_1+m_2),m_1+3m_2+1\}$.

\end{enumerate}

Observe that
$A=\big|\frac{r_1r_2N_F(R_1,R_2,g_1,g_2)}{\theta(R_1)\theta(R_2)\Theta(g_1)\Theta(g_2)} - |\mathbb{F}_{q^n}^* \backslash S_F| \big|$ is bounded by
$$
 \underbrace{
\Big| \mathop{\int}_{
	\substack{
		{d_1}_{(r_1)}|R_1r_1\\
		{d_2}_{(r_2)}|R_2r_2 }  }
\mathop{\int}_{
	\substack{h_1|\widetilde{g_1} \\ h_2| \widetilde{g_2}}	}
\sum_{
	\substack{
		\psi_ 1 \in \hat{f}_1^{-1}(\chi_{h_1})  \\
		\psi_ 2 \in \hat{f}_2^{-1}(\chi_{h_2})
	}
}
S(\eta_{d_1},\eta_{d_2},\psi_1,\psi_2)
\Big|,
 }_{\text{all are non-trivial simultaneously}} 
$$
and we also have
\begin{equation}\label{bound}
|S(\eta_{d_1},\eta_{d_2},\psi_1,\psi_2)|\leq Mq^{n/2}
\end{equation}
in all cases treated above, except in the case (i).

Let
$$
S_1=
\Big|
\mathop{\int}_{
	\substack{
		{d_1}_{(r_1)}|R_1r_1\\
		{d_2}_{(r_2)}|R_2r_2 }  }
\sum_{
	\substack{
		\psi_ 1 \in \hat{f}_1^{-1}(\chi_1)  \\
		\psi_ 2 \in \hat{f}_2^{-1}(\chi_1) \\
		(\psi_1,\psi_2) \neq (\chi_1,\chi_1)
	}
}
S(\eta_{d_1},\eta_{d_2},\psi_1,\psi_2)
\Big|,
$$
where $\chi_1$ is the trivial additive character,
and
$$
S_2=
\Big|
\mathop{\int}_{
	\substack{
		{d_1}_{(r_1)}|R_1r_1\\
		{d_2}_{(r_2)}|R_2r_2 }  }
\mathop{\int}_{
	\substack{h_1|\widetilde{g_1} \\ h_2| \widetilde{g_2}\\ (h_1,h_2)\neq (1,1)}	}
\sum_{
	\substack{
		\psi_ 1 \in \hat{f}_1^{-1}(\chi_{h_1})  \\
		\psi_ 2 \in \hat{f}_2^{-1}(\chi_{h_2})
	}
}
S(\eta_{d_1},\eta_{d_2},\psi_1,\psi_2)
\Big|,
$$
therefore $A \leq S_1+S_2$.

From inequality \eqref{bound}, Lemma \ref{lemmacohen2022}, Lemma \ref{case-b} and using
that there are $\phi(d_i)$ multiplicative characters of order $d_i$ ($i\in\{ 1,2\}$),
we get
$$
S_1\leq M q^{n/2} (q^{k_1+k_2}-1)r_1r_2W(R_1)W(R_2).
$$

From inequality \eqref{bound}, Lemma \ref{lemmacohen2022}, Lemma \ref{case-b}, using
that 
there are $\Phi_q(h_i)$ additive characters of $\mathbb{F}_q$-order $h_i$ ($i\in\{ 1,2\}$) 
we get
$$
S_2\leq M q^{n/2} q^{k_1+k_2}r_1r_2W(R_1)W(R_2) \left( W(\widetilde{g}_1)W(\widetilde{g}_2)-1\right).
$$

Putting all these inequalities together and using the case (i), we get
\begin{eqnarray*}
\frac{r_1r_2N_F(R_1,R_2,g_1,g_2)}{\theta(R_1)\theta(R_2)\Theta(g_1)\Theta(g_2)} & \geq &
q^n - (m_1+m_2+1) - (S_1+S_2) \\
& = & q^n - M q^{n/2+k_1+k_2}r_1r_2W(R_1)W(R_2) W(\widetilde{g}_1)W(\widetilde{g}_2)  \\
&   &  + M q^{n/2} r_1r_2W(R_1)W(R_2)- (m_1+m_2+1) \\
& > & q^n - M q^{n/2+ k_1+k_2}r_1r_2W(R_1)W(R_2) W(\widetilde{g}_1)W(\widetilde{g}_2).
\end{eqnarray*}

Thus, if
$$
q^{\frac{n}{2}-k_1-k_2}\geq M r_1r_2W(R_1)W(R_2) W(\widetilde{g}_1)W(\widetilde{g}_2),
$$
then $N_F(R_1,R_2,g_1,g_2)>0$.
In particular, this implies the last sentence of the theorem.
\end{proof}

\subsection{The prime sieve} 
The aim of the section is to relax further the condition
of Theorem \ref{principal}. The sieving technique from the next two results is similar to others which have
appeared in previous works about primitive and normal elements.

\begin{lemma}\label{lema-sieve}
Let $\ell_1$ be a divisor of $R_1$, let $\{p_1,\ldots,p_u\}$ be the set of all primes which
divide $R_1$ but do not divide $\ell_1$, let $\ell_2$ be a divisor of $R_2$ and let $\{q_1,\ldots,q_v\}$ be the set of all primes which
divide $R_2$ but do not divide $\ell_2$.
Also, let 
$\{P_1,\ldots,P_s\}$ be the set of all monic irreducible polynomials which divide $x^n-1$ but do not divide $g_1$, and $\{Q_1,\ldots,Q_t\}$ be the set of all monic irreducible polynomials which divide $x^n-1$ but do not divide $g_2$. Then
\begin{align}\label{desig-sieve}
N_F(R_1,R_2,x^n-1) & \geq \sum_{i=1}^u N_F(\ell_1p_i,\ell_2,g_1,g_2) +\sum_{i=1}^v N_F(\ell_1,\ell_2q_i,g_1,g_2) + \\ 								\nonumber
				   & + \sum_{i=1}^s N_F(\ell_1,\ell_2,g_1P_i,g_2) + \sum_{i=1}^t N_F(\ell_1,\ell_2,g_1,g_2Q_i) \\ 									\nonumber
				   & -(u+v+s+t-1)N_F(\ell_1,\ell_2,g_1,g_2).
\end{align}
\end{lemma}

\begin{proof}
The left hand side of \eqref{desig-sieve} counts 
the numbers of triples $(\alpha, \beta_1, \beta_2) \in \mathbb{F}_{q^n}^* \times (\mathbb{F}_{q^n})^2$ 
such that $\alpha$ is $(R_1,r_1)$-free, $F(\alpha)$ is $(R_2,r_2)$-free,
$\beta_1,\beta_2$ are normal element,
$\alpha=f_1 \circ \beta_1$ and
$F(\alpha)=f_2 \circ \beta_2$.
Observe that for such a triple $(\alpha, \beta_1, \beta_2)$, we also have that
$\alpha$ is  $(\ell_1,r_1)$-free and $(\ell_1p_i,r_1)$-free for all $i \in \{ 1,\ldots , u\}$, 
$F(\alpha)$ is  $(\ell_2,r_2)$-free and $(\ell_2q_i,r_2)$-free for all $i \in \{ 1,\ldots , v\}$,
$\beta_1$ is $g_1$-free and $g_1P_i$-free for all $i \in \{ 1,\ldots , s\}$, and
$\beta_2$ is $g_2$-free and $g_2Q_i$-free for all $i \in \{ 1,\ldots , t\}$, so that
$(\alpha, \beta_1, \beta_2)$ is  counted $u+v+s+t - (u+v+s+t-1)=1$
time on the right side of \eqref{desig-sieve}.
For any other 
triple $(\alpha, \beta_1, \beta_2) \in \mathbb{F}_{q^n}^* \times (\mathbb{F}_{q^n})^2$,
we have that either $\alpha$ is not $(\ell_1p_i,r_1)$-free for some
$i \in \{ 1,\ldots , u\}$, or $F(\alpha)$ is not $(\ell_2q_i,r_2)$-free 
for some $i \in \{ 1,\ldots , v\}$, or
$\beta_1$ is not $g_1P_i$-free for some $i \in \{ 1,\ldots , s\}$, or
$\beta_2$ is not $g_2Q_i$-free for some $i \in \{ 1,\ldots , t\}$, or
$\alpha \neq f_1 \circ \beta_1$, or $F(\alpha) \neq f_2 \circ \beta_2$,
thus this triple will not be counted in at least one term of one of the four sums of
the right hand side of \eqref{desig-sieve}.
\end{proof}

\begin{proposition}\label{prop1-crivo}  
Assume the notation and conditions of Lemma \ref{lema-sieve} 
with $R_i =\frac{q^n-1}{r_i}$ for $i \in \{ 1,2\}$. Assume also that 
the polynomials of the set $\{P_1,\ldots,P_s\}$ divide $\frac{x^n-1}{f_1}$ 
and the polynomials of the set $\{Q_1,\ldots,Q_t\}$ divide $\frac{x^n-1}{f_2}$.
Let $\delta=1-\sum_{i=1}^u \frac{1}{p_i}-\sum_{i=1}^v \frac{1}{q_i}-\sum_{i=1}^s \frac{1}{q^{\deg(P_i)}}-\sum_{i=1}^t \frac{1}{q^{\deg(Q_i)}}>0$ 
and $\Delta=2+\frac{u+v+s+t-1}{\delta}$. Denote also
$\widetilde{g_i}=\gcd (g_i, \frac{x^n-1}{f_i})$ for $i \in \{ 1,2\}$.
If 
\begin{equation}\label{crivo}
q^{\frac{n}{2}-k_1-k_2}\ge  Mr_1r_2 W(\ell_1)W(\ell_2)W (\widetilde{g_1}) W(\widetilde{g_2}) \Delta,
\end{equation}
then $N_F(R_1,R_2,x^n-1)>0$.
\end{proposition}

\begin{proof}
We can rewrite inequality \eqref{desig-sieve} in the form
\begin{align}\label{eq1-crivo}
N_F(R_1,R_2,x^n-1) & \geq \sum_{i=1}^u \left( N_F(\ell_1p_i,\ell_2,g_1,g_2)- \theta(p_i) N_F(\ell_1,\ell_2,g_1,g_2) \right) + \\ \nonumber
				   & + \sum_{i=1}^v \left( N_F(\ell_1,\ell_2q_i,g_1,g_2)- \theta(q_i) N_F(\ell_1,\ell_2,g_1,g_2) \right) + \\ \nonumber
				   & + \sum_{i=1}^s \left( N_F(\ell_1,\ell_2,g_1P_i,g_2)- \Theta(P_i) N_F(\ell_1,\ell_2,g_1,g_2) \right) + \\ \nonumber
				   & + \sum_{i=1}^t \left( N_F(\ell_1,\ell_2,g_1,g_2Q_i)- \Theta(Q_i) N_F(\ell_1,\ell_2,g_1,g_2) \right) + \\ \nonumber				   
				   & + \delta N_F(\ell_1,\ell_2,g_1,g_2).
\end{align}
From the proof of Theorem \ref{principal}, taking into account that $\theta$ is a multiplicative function
and calling $S_{\eta,\psi}:= S(\eta_{d_1},\eta_{d_2},\psi_1,\psi_2)$,
we get that $N_F(\ell_1p_i,\ell_2,g_1,g_2)$ is equal to
\begin{eqnarray*}
\frac{\theta(\ell_1)\theta(p_i)\theta(\ell_2)\Theta(g_1)\Theta(g_2)}{r_1r_2 }
\mathop{\int}_{
	\substack{
		{d_1}_{(r_1)}|\ell_1p_ir_1\\
		{d_2}_{(r_2)}|\ell_2r_2 }  }
\mathop{\int}_{
	\substack{h_1|\widetilde{g_1} \\ h_2| \widetilde{g_2}}	}
\sum_{
	\substack{
		\psi_ 1 \in \hat{f}_1^{-1}(\chi_{h_1})  \\
		\psi_ 2 \in \hat{f}_2^{-1}(\chi_{h_2})
	}
}
S_{\eta,\psi}
= \theta(p_i) N_F(\ell_1,\ell_2,g_1,g_2) \\
 + \frac{\theta(\ell_1)\theta(p_i)\theta(\ell_2)\Theta(g_1)\Theta(g_2)}{r_1r_2 }
\mathop{\int}_{
	\substack{
		{d_1}_{(r_1)}|\ell_1p_ir_1\\
		p_i | {d_1}_{(r_1)} \\
		{d_2}_{(r_2)}|\ell_2r_2 }  }
\mathop{\int}_{
	\substack{h_1|\widetilde{g_1} \\ h_2| \widetilde{g_2}}	}
\sum_{
	\substack{
		\psi_ 1 \in \hat{f}_1^{-1}(\chi_{h_1})  \\
		\psi_ 2 \in \hat{f}_2^{-1}(\chi_{h_2})
	}
}
S_{\eta,\psi},
\end{eqnarray*}
for all $i\in \{ 1,\ldots,u\}$.
%
%
From cases (ii), (iii), (iv), (v) and (vi) of the proof of Theorem \ref{principal}, 
if $\eta_{d_1}$ is not the trivial multiplicative character, then 
we have that $|S_{\eta,\psi}| \leq M q^{n/2}$ and from 
Lemma \ref{lemmacohen2022}, Lemma \ref{case-b},
we get
$$
\Big| \mathop{\int}_{
	\substack{
		{d_1}_{(r_1)}|\ell_1p_ir_1\\
		p_i | {d_1}_{(r_1)} \\
		{d_2}_{(r_2)}|\ell_2r_2 }  }
\mathop{\int}_{
	\substack{h_1|\widetilde{g_1} \\ h_2| \widetilde{g_2}}	}
\sum_{
	\substack{
		\psi_ 1 \in \hat{f}_1^{-1}(\chi_{h_1})  \\
		\psi_ 2 \in \hat{f}_2^{-1}(\chi_{h_2})
	}
}
S_{\eta,\psi} \Big| 
\leq M q^{\frac{n}{2}+k_1+k_2}r_1r_2 W(\ell_1)W(\ell_2)W(\widetilde{g_1})W(\widetilde{g_2}),
$$
hence 
\begin{align*}
|N_F(\ell_1p_i,\ell_2,g_1,g_2)- \theta(p_i) N_F(\ell_1,\ell_2,g_1,g_2)| & \leq 
\frac{\theta(\ell_1)\theta(p_i)\theta(\ell_2)\Theta(g_1)\Theta(g_2)}{r_1r_2} \times \\
 & M q^{\frac{n}{2}+k_1+k_2}r_1r_2 W(\ell_1)W(\ell_2)W(\widetilde{g_1})W(\widetilde{g_2}).
\end{align*}

Analogously, for $i\in \{1,\ldots,v\}$, we can show that 
\begin{align*}
|N_F(\ell_1,\ell_2q_i,g_1,g_2)- \theta(q_i) N_F(\ell_1,\ell_2,g_1,g_2)| & \leq 
\frac{\theta(\ell_1)\theta(q_i)\theta(\ell_2)\Theta(g_1)\Theta(g_2)}{r_1r_2} \times \\
 & M q^{\frac{n}{2}+k_1+k_2}r_1r_2 W(\ell_1)W(\ell_2)W(\widetilde{g_1})W(\widetilde{g_2}).
\end{align*}
Also, for $i\in \{1,\ldots,s\}$, we have
\begin{align*}
|N_F(\ell_1,\ell_2,g_1P_i,g_2)- \Theta(P_i) N_F(\ell_1,\ell_2,g_1,g_2)| & \leq 
\frac{\theta(\ell_1)\theta(\ell_2)\Theta(P_i)\Theta(g_1)\Theta(g_2)}{r_1r_2} \times \\
 & M q^{\frac{n}{2}+k_1+k_2}r_1r_2 W(\ell_1)W(\ell_2)W(\widetilde{g_1})W(\widetilde{g_2}),
\end{align*}
since $P_i \mid \frac{x^n-1}{f_1}$.
Finally, for $i\in \{ 1,\ldots,t\}$, we get
\begin{align*}
|N_F(\ell_1,\ell_2,g_1,g_2Q_i)- \Theta(Q_i) N_F(\ell_1,\ell_2,g_1,g_2)| & \leq 
\frac{\theta(\ell_1)\theta(\ell_2)\Theta(Q_i)\Theta(g_1)\Theta(g_2)}{r_1r_2} \times \\
 & M q^{\frac{n}{2}+k_1+k_2}r_1r_2 W(\ell_1)W(\ell_2)W(\widetilde{g_1})W(\widetilde{g_2}),
\end{align*}
since $Q_i \mid \frac{x^n-1}{f_2}$. Combining the inequalities above in \eqref{eq1-crivo}, we obtain
\begin{align*}
N_F(R_1,R_2,x^n-1) & \geq \delta N_F(\ell_1,\ell_2,g_1,g_2) - M q^{\frac{n}{2}+k_1+k_2}r_1r_2 W(\ell_1)W(\ell_2)W(\widetilde{g_1})W(\widetilde{g_2}) \times \\
& \frac{\theta(\ell_1)\theta(\ell_2)\Theta(g_1)\Theta(g_2)}{r_1r_2}
\left(
\sum_{i=1}^u \theta(p_i) + \sum_{i=1}^v \theta(q_i) + \sum_{i=1}^{s} \Theta(P_i) + \sum_{i=1}^{t} \Theta(Q_i)
\right).
\end{align*}

Therefore, from the proof of Theorem \ref{principal}, we have
\begin{align*}
N_F(R_1,R_2,x^n-1) & > \delta \theta 
\left(
q^n- M q^{\frac{n}{2}+k_1+k_2}r_1r_2 W(\ell_1)W(\ell_2)W(\widetilde{g_1})W(\widetilde{g_2})
\right)
 \\
& - M \theta q^{\frac{n}{2}+k_1+k_2}r_1r_2 W(\ell_1)W(\ell_2)W(\widetilde{g_1})W(\widetilde{g_2})
\times \\
& \left(
\sum_{i=1}^u \theta(p_i) + \sum_{i=1}^v \theta(q_i) + \sum_{i=1}^{s} \Theta(P_i) + \sum_{i=1}^{t} \Theta(Q_i)
\right) \\
& = \delta \theta 
\left(
q^n- M q^{\frac{n}{2}+k_1+k_2}r_1r_2 W(\ell_1)W(\ell_2)W(\widetilde{g_1})W(\widetilde{g_2}) \Delta
\right),
\end{align*}
where $\theta= \frac{\theta(\ell_1)\theta(\ell_2)\Theta(g_1)\Theta(g_2)}{r_1r_2}$.
Therefore we get the desired result.
\end{proof}

\section{A particular case}

In this section, we are going to deal with the particular case
where we want to determine the pairs $(q,n)$ such that
there exists a $2$-primitive $2$-normal element $\alpha \in \mathbb{F}_{q^n}$  such that $F(\alpha)$ 
is $3$-primitive
$1$-normal for all $F(x)\in \Upsilon_q(2,1)$.
Thus, from now on, $m_1=2, m_2=1, r_1=2, r_2=3, k_1=2$ and $k_2=1$. Also, all the procedures and numerical calculations are done using SageMath \cite{SAGE}.

Before studying the existence of such elements, we will show a result that will help to bound the function $W(x^n-1)$.

\begin{lemma}[\cite{ABS}, Lemma 4.3]\label{lem-part-case}
The number of monic irreducible factors of $x^n-1$ over $\mathbb{F}_q$ is less than or equal to $\frac{n}{a}+b$, where $(a,b)$ can be chosen among the following pairs:
$$
(1,0), \ \ \left(2, \frac{q-1}{2} \right), \ \ 
\left(3, \frac{q^2+3q-4}{6} \right),
$$

$$
\left(4, \frac{q^3+3q^2+5q-9}{12} \right) \ \ 
\left(5, \frac{3q^4+8q^3+15q^2+22q-48}{60} \right).
$$
\end{lemma}

Observe that $\mathbb{F}_{q^n}$ is a finite field of characteristic $p\ge 5,$ since $q^n-1$ must be a
multiple of $2$ and $3$. Observe also that the condition for the existence of a $2$-normal element 
is $\gcd(q^3 -q,n)\ne 1$ (see \cite[Lemma 3.1]{AN}).

Let $A$ be the set of pairs $(q,n)$ such that there exists a $2$-primitive $2$-normal element $\alpha \in \mathbb{F}_{q^n}$ with $F(\alpha)$ being a $3$-primitive $1$-normal element. Hence, from \cite[Lemma 3.1]{AN}, if
$(q,n)\in A$, then $6 \mid q^n-1$ and
$\gcd(q^3 -q,n)\ne 1$.

\begin{proposition}\label{firstbound}
If $6 \mid q^n-1$,
$\gcd(q^3 -q,n)\ne 1$, $n \geq 12$ and $q^n \geq 6.18 \cdot 10^{718}$, then $(q,n) \in A$.
\end{proposition}

\begin{proof}
From Theorem \ref{principal}, we get $M=6$ and the condition
\begin{equation}\label{part-cond1}
q^{\frac{n}{2}-3} \geq 36 \ W(\frac{q^n-1}{2}) W(\frac{q^n-1}{3}) W(\frac{x^n-1}{f_1}) W(\frac{x^n-1}{f_2})
\end{equation}
guarantees $(q,n) \in A$.
Let $m$ be a positive integer and $\mathcal{P}_m$ be the product of the first $m$ prime numbers.
From \cite[Lemma 4.1]{ABS} with $m=265$, we have that if $q^n \geq 3 \mathcal{P}_m= 3 \cdot 2.06 \cdot 10^{718}$, then $W(\frac{q^n-1}{r_i})< (\frac{q^n}{r_i})^{\frac{1}{9}}$ for $i\in\{ 1,2\}$. So that, using the trivial bounds $W(\frac{x^n-1}{f_1}) \leq 2^{n-2}$ and $W(\frac{x^n-1}{f_2}) \leq 2^{n-1}$, the inequality
\begin{equation}\label{part-cond2}
q^{\frac{n}{2}-3} \geq 6^{2-\frac{1}{9}} q^{\frac{2n}{9}} 2^{2n-3}
\end{equation}
implies \eqref{part-cond1}. Applying logarithms and
using that $q^n \geq 3 \mathcal{P}_m$ we have that the inequality
$$
\frac{5}{18}- \frac{\log 4}{\log q}- \frac{3}{n} \geq \frac{\log 6^{2-\frac{1}{9}} - \log 8}{\log 3 \mathcal{P}_m} \approx 0.00079
$$
implies \eqref{part-cond2}. We define $y(n,q):=\frac{5}{18}- \frac{\log 4}{\log q}- \frac{3}{n}$ and we study
the inequality $y(n,q) \geq 0.00079$. 
Note that $(3 \mathcal{P}_m)^{\frac{1}{n}} \le 211$ if and only if $n \ge \frac{\log (3\mathcal{P}_m)}{\log 211} \approx 309.25,$ therefore if 
$n \geq 310$ and $q \geq 211$,
then $q \ge (3 \mathcal{P}_m)^{\frac{1}{n}}$ and
$y(q,n) \geq \frac{5}{18}- \frac{\log 4}{\log 211}- \frac{3}{310}
\approx 0.00079,$ implying that $(q,n) \in A$.

Suppose now that $n \leq 309$, we have $y(q,n) \geq \frac{5}{18}- \frac{n\log 4}{\log (3\mathcal{P}_m)}- \frac{3}{n}$. 
We have that the right-hand side of the last inequality is greater than or equal to $0.00079$ for $12 \leq n \leq 310$. We conclude that $(q,n) \in A$ for $q \geq 211$, $n \geq 12$ and $q^n \geq 3\mathcal{P}_m$.

Now we consider $q<211$, which means that $q \leq 199$ since  $q$ is an odd prime power. From Lemma \ref{lem-part-case}, we get $W(x^n-1) \leq 2^{\frac{n}{2}+\frac{q-1}{2}}$. Thus, the inequality $q^{\frac{n}{2}-3} \geq 6^{2-\frac{1}{9}} q^{\frac{2n}{9}} 2^{n+q-4}$ is a sufficient condition for \eqref{part-cond1}. Using that 
$q^n > 3\mathcal{P}_m$, the last inequality holds if
$$
z(q,n):= \frac{5}{18}- \frac{\log 2}{\log q} - \frac{1}{n} \left( 3+q \frac{\log 2}{\log q} \right) \geq \frac{\log (6^{2-\frac{1}{9}})-\log 16}{\log (3 \mathcal{P}_m)} \approx 0.00037.
$$
Since $q^n>3 \mathcal{P}_m$, we get $n \geq \frac{\log (3\mathcal{P}_m)}{\log q}$ and $z(q,n) \geq \frac{5}{18}- \frac{\log 2}{\log q} - \frac{3 \log q}{\log (3 \mathcal{P}_m)} - \frac{q \log 2}{\log (3\mathcal{P}_m)}$. Therefore, we need to prove that the right hand side of the last inequality is greater than or equal to $0.00037$, which is true for all odd prime power 
$q$ such that $17 \leq q \leq 199$.

For the case $q \leq 13$, we use the estimate $W(x^n-1) \leq 2^{\frac{n}{3}+ \frac{q^2+3q-4}{6}}$ (see Lemma \ref{lem-part-case}) and we proceed analogously to the previous case, showing that 
\eqref{part-cond1} holds for $q=13,11,7$. Finally, for $q=5$, we use the estimate $W(x^n-1) \leq 2^{\frac{n}{4}+\frac{q^3+3q^2+5q-9}{12}}=2^{\frac{n}{4}+18}$ from Lemma \ref{lem-part-case} obtaining that a sufficient condition for \eqref{part-cond1} is $q^{\frac{5n}{18}-3} \geq 6^{2-\frac{1}{9}} 2^{\frac{n}{2}+33}$ which holds for $n \geq 1029$ (because $q^n \ge 3 \mathcal{P}_m$). This completes the proof. 
\end{proof}

Before continuing, we need to explain the procedures in Appendix A.
Procedure \textbf{\ref{At}} 
calculates the constant $A_t$
used in  the bound $W(M) \leq A_t \cdot M^{\frac{1}{t}}$ (see \cite[Lemma 2.9]{AN}).
We bound $W(\frac{q^n-1}{r_i})$ by $r_i^{-\frac{1}{t}} \cdot A_t \cdot q^{\frac{n}{t}}$
($i\in\{ 1,2\}$) in the condition of the last sentence of Theorem \ref{principal} and we get that if
$6 \mid (q^n-1)$, $\gcd(q^3 -q,n)>1$ and
\begin{equation}\label{cond1-crivo-particular}
	q^{\frac{n}{2}-3} \geq 6^{2-\frac{1}{t}} \cdot A_t^2 \cdot q^{\frac{2n}{t}} \cdot W \left( \frac{x^n-1}{f_1} \right) W \left( \frac{x^n-1}{f_2} \right),
\end{equation} 
then $(q,n) \in A$.
Let us explain
procedure \textbf{\ref{valw1w2}}.
If $\gcd(q,n)>1$ then $W(\frac{x^n-1}{f_i})=W(x^n-1)$ for $i \in \{1,2\}$. If $\gcd(q,n)=1$
and
$\gcd(q-1,n)>1$, then $x^n-1$ has at least two linear factors, so that we may choose $f_1$ as the product
of two linear factors and $f_2=x-1$.
If $\gcd(q,n)=1$,
$\gcd(q-1,n)=1$ and $\gcd(q+1,n)>1$, then $x^n-1$ has only one linear factor and at least
one monic irreducible factor of degree $2$. Thus, \textbf{\ref{valw1w2}($q,n$)}
returns the pair $(w_1,w_2)$ where $W(\frac{x^n-1}{f_i}) = 2^{w_i}$. 
Therefore
procedure \textbf{\ref{testAt}} verifies
if \eqref{cond1-crivo-particular} holds.

The value of $\Delta$ given in \textbf{\ref{specialsieve}} 
line \ref{Delta} is greater than or equal to
$\Delta$ from Proposition \ref{prop1-crivo}, since the pair $(S,u_0)$
given by procedure \textbf{\ref{sumfactors}} 
satisfies $u \le u_0 $, $v \le u_0$ and $\delta \ge 1-2S$ for
$(u,v,\delta)$ in Proposition \ref{prop1-crivo}. Thus, if \textbf{\ref{specialsieve}($q,n,p_0$)} 
returns True, then \eqref{crivo} holds.

Finally, procedure \textbf{\ref{totalsieve}}, 
with the auxiliary procedures \textbf{\ref{monicfactors}} 
and \textbf{\ref{listsieve}},
verifies if \eqref{crivo} holds for some choice of $\ell_1$, $\ell_2$, $g_1$ and $g_2$.

\begin{proposition}\label{n12}
If $6 \mid (q^n-1)$,
$\gcd(q^3 -q,n)\ne 1$, $q \geq 5$ and $n \geq 12$, then $(q,n) \in A$.
\end{proposition}

\begin{proof}
From the previous proposition, if $q^n \geq 6.18 \cdot 10^{718}$, then $(q,n) \in A$.
So that, also assume  $q^n < 6.18 \cdot 10^{718}$.
Let $p_0$ be a fixed prime number.
Let $\ell_i$ be the product of prime numbers $p<p_0$ which divide $\frac{q^n-1}{r_i}$.
Denote by $m$ the number of primes less than $p_0$ which divide $q^n-1$.
Then $2\le m \le \pi(p_0-1)$ 
and $W(\ell_i)\le 2^m$. Let also $\{ p_1, \ldots , p_u \}$ be the set of primes
which divide $q^n-1$ and are greater than or equal to $p_0$ and
denote by $\mathcal{P}(u,p_0)$ and by $\mathcal{S}(u,p_0)$ the product and the sum of the inverses, respectively, of the first $u$ prime numbers greater than or equal to $p_0$.
Hence 
$$\mathcal{P}(u,p_0) \le p_1 \cdots  p_u  \ \ \ \text{and} \ \ \ 
\sum_{i=1}^u \frac{1}{p_i} \le \mathcal{S}(u,p_0).
$$
Let $u(m)=\max \{ u \mid \mathcal{P}_m \cdot \mathcal{P}(u,p_0) \leq  \mathcal{P}\}$,
where $\mathcal{P}=6.18 \cdot 10^{718}$ and
$\mathcal{P}_m$ is defined as in Proposition \ref{firstbound}. 
From Proposition \ref{prop1-crivo}
with $\ell_1$ and $\ell_2$ as given in the last paragraph and $g_1=g_2=x^n-1$, we get
$\delta \geq 1-2\mathcal{S}(u(m),p_0)$ and $\Delta \leq 2 + \frac{2u(m)-1}{1-2\mathcal{S}(u(m),p_0)}=: \Delta(m)$. 
We need to choose $p_0$ such that $1-2\mathcal{S}(u(m),p_0)>0$ in order to use
Proposition \ref{prop1-crivo}. Thus, if $q^{\frac{n}{2}-3} \ge 6^2 \cdot 2^{2m} \cdot 2^{2n-3} \cdot \Delta(m)$
for $2 \le m \le \pi(p_0)-1$, then
\eqref{crivo} holds. 
If we suppose that $q \geq q_0$, 
the last condition holds if
\begin{equation}\label{sieve1}
(q^n)^{\frac{1}{4}-\log_{q_0}4} \geq \max \{ 6^2 \cdot 2^{2m-3} \cdot \Delta(m) \mid 2 \le m \le \pi(p_0)-1\},
\end{equation}
since
$\frac{n}{2}-3 \ge \frac{n}{4}$. For $p_0=89$ and $q_0=10009$, we get
$1-2\mathcal{S}(u(m),p_0)>0$ for all $2 \le m \le 23$, and \eqref{sieve1} holds for
$q^n \geq 1.15 \cdot 10^{190}$. Repeating the process with
$q_0=10009$ and $(p_0,\mathcal{P})=(41,1.15 \cdot 10^{190})$, $(31,8.31 \cdot 10^{118})$, $(29,3.00 \cdot 10^{100})$,
$(29,5.31 \cdot 10^{94})$, $(29,9.01 \cdot 10^{92})$
sequentially,
we obtain the bound $q^n \geq 1.66 \cdot 10^{92}$. Analogously, for $q_0=10^5+3$ and
$(p_0,\mathcal{P})=(29, 1.66 \cdot 10^{92})$, $(23, 6.42 \cdot 10^{70})$,
$(23, 2.33 \cdot 10^{65})$, $(23, 4.10 \cdot 10^{63})${\color{blue},}
we get $q^n \geq 1.29 \cdot 10^{63}$. Therefore, if
\begin{equation}\label{cota-crivo1}
q \geq \mathcal{M}_n:=\min \left\{ \begin{array}{lcc}
              \max \{ 10^5+3, (1.29 \cdot 10^{63})^{1/n} \} \\
              \max \{ 10009, (1.66 \cdot 10^{92})^{1/n} \} \\
              (6.18 \cdot 10^{718})^{1/n},
             \end{array}
   \right. 
\end{equation}
$6 \mid q^n-1$, $\gcd(q^3 -q,n)\ne 1$, $q \geq 5$ and $n \geq 12$, we have $(q,n) \in A$.

In order to prove the proposition, we only need to consider $n \le 1028$, since for $n \geq 1029$ we have $(6.18 \cdot 10^{718})^{\frac{1}{n}}<5$.
procedure \textbf{\ref{testAt}($q,n,t$)}
holds for $12 \leq n \leq 1028$, $5 \leq q \le \mathcal{M}_n$ and $t=8$,
with $6 \mid q^n-1$ and $\gcd(q^3 -q,n)\ne 1$,
except for $67065$ pairs $(q,n)$.
For such elements, \textbf{\ref{specialsieve}($q,n,p_0$)} 
holds with $p_0=71$
if $q^n > 10^{100}$, $p_0=53$ if $10^{30} < q^n \le 10^{100}$, and $p_0=23$ if $q^n\le 10^{30}$ except for $1915$ pairs.
Finally, \textbf{\ref{totalsieve}($q,n$)} 
holds
for all the remaining cases.
\end{proof}

\begin{proposition}\label{casen11-8}
We have $(q,n) \in A$ if one of the following holds.
\begin{enumerate}
\item[(a)]  $6 \mid q^n-1$,
$\gcd(q^3 -q,n)\ne 1$, $q \geq 5$ and $n\in \{ 10,11\}$.
\item[(b)] $6 \mid (q-1)$, $q \ge 4.413\cdot 10^{9}$ and $n=9$.
\item[(c)] $\gcd(q,6)=1$,  $q \ge 6.515\cdot 10^{14}$ and $n=8$.
\end{enumerate}
\end{proposition}

\begin{proof}
From \cite[Lemma 4.1]{ABS}, 
if $m$ is a positive integer and
$N$ is a real number satisfying $\frac{m\log(2)}{\log(\mathcal{P}_m)}<\frac{1}{N}$, then
$W\left( \frac{q^n-1}{r_i} \right) \leq \left( \frac{q^n}{r_i} \right)^{\frac{1}{N}}$ 
for $i \in \{1,2\}$ and
for all prime power $q$
such that $\frac{q^n}{3}\ge \mathcal{P}_m$. If we also consider
$W(\frac{x^n-1}{f_1})\le 2^{n-2}$ and 
$W(\frac{x^n-1}{f_2})\le 2^{n-1}$, then
\begin{equation}\label{intermediario}
q^{\frac{n}{2}-3} \geq 6^{2-\frac{1}{N}}\cdot  q^{\frac{2n}{N}} \cdot 2^{2n-3}
\end{equation}
implies \eqref{part-cond1}. Observe that \eqref{intermediario} is equivalent to
$$
q \ge \left( 6^{2-\frac{1}{N}}\cdot  2^{2n-3}\right)^{\frac{2N}{Nn-4n-6N}}.
$$
Since $\frac{q^n}{3}\ge \mathcal{P}_m$, then \eqref{part-cond1} holds if
\begin{equation}\label{intermediario2}
3\mathcal{P}_m \ge \left( 6^{2-\frac{1}{N}}\cdot  2^{2n-3}\right)^{\frac{2Nn}{Nn-4n-6N}}.
\end{equation}
For $8 \le n \le 11$, Table \ref{n8to11} shows the values of $N$ and $m$
for which  \eqref{intermediario2} holds for all prime power $q$ such that $q^n \ge 3\mathcal{P}_m$.
\begin{table}[h]
\centering
\begin{tabular}{cccc}
$n$  & $N$ & $m$ & $3\mathcal{P}_m$  \\
\hline
$11$ & $9.161$  & $291$ &   $\sim 2.717\cdot 10^{803}$\\
$10$ & $10.206$  & $534$ & $\sim  3.819\cdot 10^{1641}$ \\
$9$ & $12.075$  & $1618$ & $\sim  1.488\cdot 10^{5882}$ \\
$8$ & $16.008$  & $18011$ & $\sim 3.980\cdot 10^{86793}$ 
\end{tabular}
\vspace*{0.5cm}
\caption{Values of $N$ and $m$ for which \eqref{intermediario2} holds.}
\label{n8to11}
\end{table}

Suppose now that $q^n < \mathcal{P}$, where $\mathcal{P}$ is the value of the
column of $3\mathcal{P}_m$ in Table \ref{n8to11}, for the respective value of $n$ and
we proceed to use the sieve method like it was done at the beginning of the proof of Proposition \ref{n12}. Then $\Delta(m)$  depends on the
prime number $p_0$ to be chosen,
and we get that
if $6 \mid q^n-1$,
$\gcd(q^3 -q,n)\ne 1$ and
\begin{equation}\label{sieve2}
	q^{\frac{n}{2}-3} \geq \max \{ 6^2 \cdot 2^{2m} \cdot 2^{2n-3} \cdot \Delta(m) \mid 2 \le m \le \pi(p_0)-1\},
\end{equation}
then $(q,n)\in A$. We use this process repeatedly for $n\in \{8,9,10,11\}$, as shown in
Table \ref{crivon8to11}.

{\scriptsize 
\begin{table}[h]
\begin{minipage}{.5\linewidth}
\centering
\begin{tabular}{ccc}
\multicolumn{3}{c}{$n=8$}\\
$\mathcal{P}$ & $p_0$ & new bound  \\
\hline
$3.980\cdot 10^{86793}$ & $1609$ &   $q^n \ge 8.261\cdot 10^{1320}$\\
$8.261\cdot 10^{1320}$  & $131$ & $q^n \ge 1.634\cdot 10^{230}$ \\
$1.634\cdot 10^{230}$   & $47$ & $q^n \ge 1.294\cdot 10^{139}$ \\
$1.294\cdot 10^{139}$   & $37$ & $q^n \ge 7.454\cdot 10^{122}$\\
$7.454\cdot 10^{122}$   & $31$  & $q^n \ge 5.975\cdot 10^{119}$\\
$5.975\cdot 10^{119}$   &  $31$ & $q^n \ge 3.242\cdot 10^{118}$
\end{tabular}
\end{minipage}%
\begin{minipage}{.5\linewidth}
\centering
\begin{tabular}{ccc}
\multicolumn{3}{c}{$n=9$}\\
$\mathcal{P}$ & $p_0$ & new bound  \\
\hline
$1.488\cdot 10^{5882}$ & $313$ &   $q^n \ge 5.923\cdot 10^{301}$\\
$5.923\cdot 10^{301}$  & $53$ & $q^n \ge 1.517\cdot 10^{115}$ \\
$1.517\cdot 10^{115}$   & $31$ & $q^n \ge 5.204\cdot 10^{91}$ \\
$5.204\cdot 10^{91}$   & $29$ & $q^n \ge 3.679\cdot 10^{87}$ \\
$3.679\cdot 10^{87}$   & $29$  & $q^n \ge 6.347\cdot 10^{86}$
\end{tabular}
\end{minipage}

\vspace*{0.5cm}

\begin{minipage}{.5\linewidth}
\centering
\begin{tabular}{ccc}
\multicolumn{3}{c}{$n=10$}\\
$\mathcal{P}$ & $p_0$ & new bound  \\
\hline
$3.819\cdot 10^{1641}$ & $149$ &   $q^n \ge 3.891\cdot 10^{160}$\\
$3.891\cdot 10^{160}$  & $41$ & $q^n \ge 5.414\cdot 10^{85}$ \\
$5.414\cdot 10^{85}$   & $29$ & $q^n \ge 1.155\cdot 10^{75}$ \\
$1.155\cdot 10^{75}$   & $23$ & $q^n \ge 2.874\cdot 10^{73}$ \\
$2.874\cdot 10^{73}$   & $23$  & $q^n \ge 8.442\cdot 10^{72}$
\end{tabular}
\end{minipage}%
\begin{minipage}{.5\linewidth}
\centering
\begin{tabular}{ccc}
\multicolumn{3}{c}{$n=11$}\\
$\mathcal{P}$ & $p_0$ & new bound  \\
\hline
$2.717\cdot 10^{803}$ & $97$ &   $q^n \ge 9.605\cdot 10^{115}$\\
$9.605\cdot 10^{115}$  & $31$ & $q^n \ge 6.726\cdot 10^{72}$ \\
$6.726\cdot 10^{72}$   & $23$ & $q^n \ge 6.673\cdot 10^{66}$ \\
$6.673\cdot 10^{66}$   & $23$ & $q^n \ge 5.224\cdot 10^{65}$ \\
$5.224\cdot 10^{65}$   & $23$  & $q^n \ge 2.574\cdot 10^{65}$
\end{tabular}
\end{minipage}%
\vspace*{0.5cm}
\caption{New bound for $n\in \{8,9,10,11\}$ using the sieving technique.}
\label{crivon8to11}
\end{table}
}

For $n=11$ and $n=10$, if $6 \mid q^n-1$,
$\gcd(q^3 -q,n)\ne 1$ and $q \ge 883933$ (for $n=11$) or $q \ge 1.962\cdot 10^{7}$ (for $n=10$), then $(q,11), (q,10) \in A$.
procedure \textbf{\ref{specialsieve}($q,n,p_0$)} 
holds with $p_0=23$ and $q<883933$ or $q<1.962\cdot 10^{7}$, except for $120$ (for $n=11$) or $7978$ (for $n=10$) pairs.
procedure \textbf{\ref{totalsieve}($q,n$)} 
holds for all the remaining cases.

For $n=9$, 
we have $\gcd(q^3 -q,9)\ne 1$ for all prime power $q$
and $6 \mid (q^9-1)$ if and only if $6 \mid (q-1)$.
Thus we get that if $6 \mid (q-1)$ and $q \ge 4.413\cdot 10^{9}$, then $(q,9)\in A$.
For $n=8$,
we have $\gcd(q^3 -q,8)\ne 1$ for all prime power $q$, 
and $6 \mid (q^8-1)$ if and only if $\mathbb{F}_q$ has characteristic greater than $3$.
Therefore we get that if $\gcd(q,6)=1$, then $(q,8)\in A$.
\end{proof}

The following lemma shows how to use procedure \textbf{\ref{boundsieve}($q,n,p_0,q_{\textrm{min}},q_{\textrm{max}}$)} 
for $n \in \{ 7,8,9\}$.
\begin{lemma}\label{proc-boundsieve}
Let $q_{\textrm{min}}$ and $q_{\textrm{max}}$ be positive integers,
$p_0$ be a prime number, and $n \in \{ 7,8,9\}$. Procedure
\textbf{\ref{boundsieve}($q_{\textrm{min}},q_{\textrm{max}},n,p_0$)}
returns a pair $(q_{\textrm{new}},B)$. If $B=\textrm{true}$,
then for $6 \mid q^n-1$,
$\gcd(q^3 -q,n)\ne 1$, $q_{\textrm{min}}\le q \le q_{\textrm{max}}$ and $q \ge q_{\textrm{new}}$,
we have $(q,n) \in A$.
\end{lemma}
\begin{proof}
Let suppose that $6 \mid q^n-1$,
$\gcd(q^3 -q,n)\ne 1$ and $q_{\textrm{min}}\le q \le q_{\textrm{max}}$.
We will prove that if $B=\textrm{true}$ and $q \ge q_{\textrm{new}}$, then 
$(q,n) \in A$.

Let $p$ the greatest prime number which divides $n$, and 
write $n=p^a n_0$, where $\gcd(n,n_0)=1$.
Define
$$
\bar{p}= \left\{ \begin{array}{ll}
	2p^a &   \textrm{if }  2 \nmid p, \\
	2^a  &   \textrm{if }  p=2.
\end{array}
\right.
$$
If $\frak{p}$ is a prime number such that $\frak{p} \mid q^n-1$ and $\frak{p} \nmid q^{\frac{n}{p}}-1$,
then $q^{n_0}$ has multiplicative order $p^a$ modulo $\frak{p}$.
This implies that $\frak{p}$ is of the form
$\bar{p}j+1$.

We will use Proposition \ref{prop1-crivo}. For $i\in\{ 1,2\}$, 
let $g_i\in \mathbb{F}_q[x]$
such that $\widetilde{g}_i=1$, and
let $\ell_i$ be the product of prime numbers which divide $\frac{q^n-1}{r_i}$ that
are less than $p_0$ and are not of the form $\bar{p}j+1$.
Denote by $m$ the number of primes less than $p_0$ which divide $q^n-1$ and are not of the form $\bar{p}j+1$. Then 
$2 \le m \le \pi_{\bar{p}}(p_0 - 1)$, where $\pi_{\bar{p}}(p_0 - 1)$
is the number of primes less than $p_0$ which are not of the form $\bar{p}j+1$.
We also have $W(\ell_i)\le 2^m$. 

Let $u_1$ be the number of primes
which divide $q^n-1$, are greater than or equal to $p_0$, and
are not of the form $\bar{p}j+1$. Let also
$u_2$ be the number of primes
which divide $q^n-1$, are greater than or equal to $p_0$ and
are of the form $\bar{p}j+1$.
Let $\mathcal{P}_0(m)$ be the product of the first $m$ prime numbers which are not of the form
$\bar{p}j+1$. Let $\mathcal{P}_1(u_1,p_0)$ and $\mathcal{S}_1(u_1,p_0)$ be the product and the sum of the inverses, respectively, of the first $u_1$ prime numbers greater than or equal to $p_0$
which are not of the form $\bar{p}j+1$.
Let also $\mathcal{P}_2(u_2)$ and $\mathcal{S}_2(u_2)$ be the product and the sum of the inverses, respectively, of the first $u_2$ prime numbers 
which are of the form $\bar{p}j+1$.

So that, since $q \le q_{\textrm{max}}$, the inequalities
\begin{equation}\label{cond-crivo}
e_1  \mathcal{P}_0(m)  \mathcal{P}_1(u_1,p_0) \le q_{\textrm{max}}^{\frac{n}{p}} -1
\textrm{ and }
e_2  \mathcal{P}_0(m) \mathcal{P}_1(u_1,p_0) \mathcal{P}_2(u_2)
\le q_{\textrm{max}}^n -1
\end{equation}
hold,
where $e_1$ and $e_2$ are defined as follows. If $n=7$, then $e_1=e_2=1$.
If $n=8$, then $e_1=2^3$ and $e_2=2^4$, since
$2 \mid q-1$, $2^3 \mid q^2-1$, $2^4\mid q^4-1$, $2^5 \mid q^8-1$, and
$2 \cdot 3$ appears in the product $\mathcal{P}_0(m)$. If $n=9$  then $e_1=3$
and $e_2=3^2$, since $3 \mid q-1$, $3^2 \mid q^3-1$ and $3^3 \mid q^9-1$.

Considering that $q \geq q_{min}$, we have
\begin{eqnarray*}
\delta \ge 1 - 2 (\mathcal{S}_1(u_1,p_0) + \mathcal{S}_2(u_2)) -\frac{2n-3}{q_{min}} =: \delta_{u_1,u_2}
\text{ and } \\
\Delta \le 2 + \frac{2(u_1+u_2)+2n-4}{\delta_{u_1,u_2}} =: \Delta_{u_1,u_2}.
\end{eqnarray*}

Let $\mathcal{U}(p_0)$ be the set of triples $(m,u_0,u_1)$ satisfying
$2 \le m \le \pi_{\bar{p}}(p_0 - 1)$ and \eqref{cond-crivo}. If $\delta_{u_1,u_2}>0$
for all $(m,u_1,u_2)\in \mathcal{U}(p_0)$, then 
\eqref{crivo} holds for all prime power $q$ such that
$q^{\frac{n}{2}-3} \ge 6^2 \cdot 2^{2m} \cdot \Delta_{u_1,u_2}$
for all $(m,u_1,u_2)\in \mathcal{U}(p_0)$. The last sentence is equivalent to
\begin{equation}\label{cond-n9-7}
q \ge q_{\textrm{new}}:=\max \left\{
(6^2\cdot 2^{2m} \cdot \Delta_{u_1,u_2})^{\frac{2}{n-6}} \mid \forall (m,u_1,u_2)\in \mathcal{U}(p_0)
\right\}.
\end{equation}
This completes the proof.
\end{proof}

\begin{proposition}\label{casen8-9}
If
$6 \mid q^n-1$
and $n\in \{ 8,9\}$,
then
$(q,n) \in A$.
\end{proposition}
\begin{proof}
In the proof of Proposition \ref{casen11-8} we have already seen that $\gcd(q^3-q,n)\neq 1$ for $n \in \{ 8,9\}$.	
Let suppose that $6 \mid q^n-1$. For $n=9$,
from  Proposition \ref{casen11-8}, if $q \ge 4.413\cdot 10^{9}$, then $(q,9)\in A$. Let suppose
that $q<4.413\cdot 10^{9}$. We now use Lemma \ref{proc-boundsieve}.
For $q_{\textrm{min}}=10^4,q_{\textrm{max}}=4.413\cdot 10^{9},n=9,p_0=19$,
procedure \textbf{\ref{boundsieve}} returns $q_{\textrm{new}}< 585229$
and $B=\textrm{true}$.
For $q_{\textrm{min}}=10^4,q_{\textrm{max}}=585229,n=9,p_0=17$,
procedure \textbf{\ref{boundsieve}} returns $q_{\textrm{new}}< 128243$
and $B=\textrm{true}$.
For $q_{\textrm{min}}=10^4,q_{\textrm{max}}=128243,n=9,p_0=13$,
procedure \textbf{\ref{boundsieve}} returns $q_{\textrm{new}}< 65337$
and $B=\textrm{true}$.
For $q_{\textrm{min}}=10^4,q_{\textrm{max}}=65337,n=9,p_0=13$,
procedure \textbf{\ref{boundsieve}} returns $q_{\textrm{new}}< 62416$
and $B=\textrm{true}$. Finally, \textbf{\ref{totalsieve}} holds for all the $3182$ prime powers $q<62416$ satisfying $6 \mid (q-1)$.

For $n=8$, from 
Proposition \ref{casen11-8}, if $q \ge 6.515\cdot 10^{14}$, then $(q,8)\in A$. Let suppose
that $q<6.515\cdot 10^{14}$. We now use Lemma \ref{proc-boundsieve}.
For $q_{\textrm{min}}=10^9,q_{\textrm{max}}=6.515\cdot 10^{14},n=8,p_0=37$,
procedure \textbf{\ref{boundsieve}} returns $q_{\textrm{new}}<6.226\cdot 10^{10}$
and $B=\textrm{true}$. Now, we use repeatedly procedure  \textbf{\ref{boundsieve}}
with $q_{\textrm{min}}=10^9,n=8$ and $p_0=29$. 
Starting with $q_{\textrm{max}}=6.226\cdot 10^{10}$,
we obtain successively $q_{\textrm{new}}<4.998\cdot 10^9$, $q_{\textrm{new}}<2.069\cdot 10^9$,
$q_{\textrm{new}}<1.921\cdot 10^9$, and $q_{\textrm{new}}<1.781\cdot 10^9$.


To reduce the bound, we will consider two cases.
Suppose first that $9 \mid q^8-1$. 
We  use procedure \textbf{\ref{boundsieve}} by modifying line 7. If
$9 \mid q^8-1$, then $q\equiv \pm 1 \pmod 9$, thus $9 \mid q^2-1$ and
$9 \mid q^4-1$.
We replace line 7 by
$e_1=2^3\cdot 3$ and $e_2=2^4 \cdot 3$.
For $q_{\textrm{min}}=10^9,q_{\textrm{max}}=1.781\cdot 10^{9},n=8,p_0=29$.
This modified procedure \textbf{\ref{boundsieve}} returns
$q_{\textrm{new}}<1.572\cdot 10^{9}$. If in addition we suppose that not all prime numbers
less than $29$ and are not  of the form $8j+1$ divide $q^8-1$,
we replace line 11 by $m \in \{ 2,\ldots, m_{max}-1\}$.
We fix $q_{\textrm{min}}=10^8,n=8,p_0=29$, and, beginning with 
$q_{\textrm{max}}=1.572\cdot 10^{9}$,
this second modified procedure \textbf{\ref{boundsieve}} returns successively
$q_{\textrm{new}}<4.453\cdot 10^{8}$,
$q_{\textrm{new}}<3.671\cdot 10^{8}$ and
$q_{\textrm{new}}<3.429\cdot 10^{8}$.
There are $2923$ prime powers
between $3.429\cdot 10^8$ and $1.572\cdot 10^{9}$ satisfying
$2^5\cdot 3^2 \cdot 5 \cdot 7 \cdot 11 \cdot 13 \cdot 19 \cdot 23=
629909280 \mid (q^8-1)$,
and
procedure \textbf{\ref{specialsieve}($q,n,p_0$)} holds
with $n=8$ and $p_0=11$ for these prime powers, except for $1564$ prime powers.
Procedure \textbf{\ref{totalsieve}($q,n$)} 
holds for all the remaining cases.

Suppose now that $9 \nmid q^8-1$. In this case $3 \nmid \ell_2$, so 
$W(\ell_2)=2^{m-1}$. So that, we replace line 22 in the original
procedure \textbf{\ref{boundsieve}} by
$Val \gets (6^2\cdot \Delta \cdot 2^{2m-1})^{\frac{2}{n-6}}$.
For $q_{\textrm{min}}=10^8,q_{\textrm{max}}=1.781\cdot 10^{9},n=8,p_0=29$,
this modified procedure \textbf{\ref{boundsieve}} returns
 successively
$q_{\textrm{new}}<8.905\cdot 10^{8}$ and
$q_{\textrm{new}}<7.341\cdot 10^{8}$. As in the previous case, if in addition we suppose
that not all prime powers
less than $29$ and which are not  of the form $8j+1$ divide $q^8-1$,
we replace line 11 by $m \in \{ 2,\ldots, m_{max}-1\}$. So,
for $q_{\textrm{min}}=10^8,q_{\textrm{max}}=7.341\cdot 10^{8},n=8,p_0=29$,
the modified procedure \textbf{\ref{boundsieve}} returns successively
$q_{\textrm{new}}<2.227\cdot 10^{8}$,
$q_{\textrm{new}}<1.715\cdot 10^{8}$, and
$q_{\textrm{new}}<1.450\cdot 10^{8}$.


We use now procedure \textbf{\ref{specialsieve}($q,n,p_0$)}
for prime powers
between $1.450\cdot 10^8$ and $7.341\cdot 10^{8}$ satisfying
$2^5\cdot 3 \cdot 5 \cdot 7 \cdot 11 \cdot 13 \cdot 19 \cdot 23=209969760 \mid (q^8-1)$
and $9 \nmid (q^8-1)$. There are $4470$ of those primes powers and
for these prime powers procedure \textbf{\ref{specialsieve}($q,n,p_0$)} holds
with $n=8$ and $p_0=11$, except for $1596$ prime powers.
Procedure \textbf{\ref{totalsieve}($q,n$)} 
holds for all the remaining cases.

There are $3421707$ odd prime powers between $1.450\cdot 10^8$ and $3.429\cdot 10^8$
satisfying $9 \mid (q^8-1)$. For these prime powers, procedure \textbf{\ref{specialsieve}($q,n,p_0$)} holds
with $n=8$ and $p_0=11$, except for $1053563$ prime powers.
Procedure \textbf{\ref{totalsieve}($q,n$)} 
holds for all these exceptions.

Finally, there are $4090405$ prime powers $q< 1.450\cdot 10^8$
with $6 \mid q-1$. For these prime powers \textbf{\ref{specialsieve}($q,n,p_0$)} holds
with $n=8$ and $p_0=7$, except for $1660117$ prime powers. Procedure \textbf{\ref{totalsieve}($q,n$)} 
holds for all the remaining cases. 
\end{proof}

\begin{proposition}\label{casen7}
For $q \geq 5.259\cdot 10^{15}$
and $n=7$, 
we have $(q,7) \in A$ if and only if
$6 \mid q-1$ and $q\equiv 0,\pm 1 \pmod 7$.
\end{proposition}
\begin{proof}
Note first that $6 \mid q^7-1$ is equivalent to $6 \mid q-1$ and 
$\gcd(q^3 - q,7)\ne 1$ is equivalent to $q\equiv 0,\pm 1 \pmod 7$. 

First we apply the sieve method from Proposition \ref{prop1-crivo} 
to obtain an initial bound. Let $\ell_i$ be the product of prime numbers $p$ which divide $\frac{q^7-1}{r_i}$ ($i\in\{ 1,2\}$)
such that $p<2^{20}$ or $p>2^{30}$. Therefore the sets $\{p_1,\ldots, p_u\}$, $\{q_1,\ldots, q_v\}$ in Lemma \ref{lema-sieve}
are equal to the set of primes dividing
$q^7-1$ which are between $2^{20}$ and $2^{30}$. 
Consider also $g_1=g_2=1$ and
$q\ge 10^9$.  We get
$$
\delta  \ge 1 - 2\mathcal{S} - \frac{s+t}{10^9}
\textrm{ and }
\Delta = 2+ \frac{u+v+s+t-1}{\delta} \le 2+\frac{2\cdot 54318003 + 6+5-1}{\delta},
$$
where $s \le 6$, $t \le 5$,
$\mathcal{S}$ is the sum of the inverses of prime numbers between $2^{20}$ and $2^{30}$
and $54318003$ is the number of primes between $2^{20}$ and $2^{30}$.

From \cite[Lemma 2.9]{AN}, we have
$$
\frac{W(l_i)}{\sqrt[30]{l_i}} = 
\prod_{{\substack{p \mid l_i \\ p \text{ is prime}}}} \frac{2}{\sqrt[30]{p}}
\le
\prod_{{\substack{p < 2^{20} \\ p \text{ is prime}}}} \frac{2}{\sqrt[30]{p}} =: B.
$$
We obtain the bound $W(l_i)\le B q^{\frac{7}{30}}$ with $B < 6.8777\cdot 10^{9530}$. 
Putting all together, we have that if
$$
q^{\frac{1}{2}} \ge 6^2 \cdot B^2 \cdot q^{\frac{14}{30}} \cdot \Delta,
$$
then \eqref{crivo} holds. Hence,
for $q \ge 8.5184\cdot 10^{572158}$, we have $(q,7)\in A$.

We use again Proposition \ref{prop1-crivo}, with $\ell_i=\frac{q-1}{r_i}$ and
$g_i$ satisfying $\gcd(g_i,\frac{x^7-1}{f_i})=1$, for $i \in \{ 1,2\}$. Therefore the sets
$\{p_1,\ldots,p_u\}$ and $\{q_1,\ldots,q_v\}$ are equal, and are composed
by prime numbers of the form $14j+1$. Denote by $\mathcal{P}_u$ and $\mathcal{S}_u$
the product and the sum of the inverses, respectively, of the first $u$ prime numbers
of the form $14j+1$. Suppose also that $q \ge 10^7$. Since $\mathcal{P}_u\le \frac{q^7-1}{q-1}$ and
$q < 8.5184\cdot 10^{572158}$, we get $u \le 476020$,
$\mathcal{S}_u\le 0.29162$,
$\delta \ge 1 - 2 \cdot 0.29162 - \frac{6+5}{10^7}=0.4167589$, and
$\Delta \le 2 + \frac{2 \cdot 476020 + 13 -1}{0.4167589}\le 2.285\cdot 10^6$.
Using that $W(\frac{q-1}{r_i}) \leq A_t \cdot (\frac{q}{r_i})^{\frac{1}{t}}$ (see \cite[Lemma 2.9]{AN}), we get that 
$$
q^{\frac{1}{2}} \ge  6^{2-\frac{1}{t}}\cdot A_t^2 \cdot q^{\frac{2}{t}} \cdot  2.285\cdot 10^6
$$
implies the condition \eqref{crivo} for some $t>4$.
With $t=7.12$, the previous inequality holds for $q \ge 3.4726\cdot 10^{58}$.

Fix $q_{\textrm{min}}=10^9$ and $n=7$. For $q_{\textrm{max}}=3.4726\cdot 10^{58}$ and $p_0=37$,
procedure \textbf{\ref{boundsieve}} returns $q_{\textrm{new}}<1.914\cdot 10^{22}$.
Now, we also fix $p_0=19$ and, for $q_{\textrm{max}}=1.914\cdot 10^{22}$, we get
successively  $q_{\textrm{new}}<2.073\cdot 10^{17}$,
$q_{\textrm{new}}<1.011\cdot 10^{16}$,
$q_{\textrm{new}}<5.565\cdot 10^{15}$, and
$q_{\textrm{new}}<5.259\cdot 10^{15}$.
%
%
\end{proof}


We may summarize the results of this section in the following theorem.
\begin{theorem}\label{principalnumerical}
	Let $q$ be a prime power, $n$ be a natural number and $F \in \Upsilon_q(2,1)$. If
	$n\ge 8$, 
	there exists 
	$2$-primitive $2$-normal element $\alpha \in \mathbb{F}_{q^n}$
	such that $F(\alpha)$ is $3$-primitive $1$-normal if and only if
	$6 \mid (q^n-1)$ and $\gcd(q^3-q,n)\neq 1$. 
	For $n=7$, if $q \geq 5.259 \cdot 10^{15}$ then
	there exists a $2$-primitive $2$-normal element $\alpha \in \mathbb{F}_{q^n}$
	such that $F(\alpha)$ is $3$-primitive $1$-normal if and only if
	$6 \mid (q-1)$ and $7 \mid (q^3-q)$.
\end{theorem}

We conjecture that $(q,7)\in A$ for all prime power $q$
such that $q \equiv 0, \pm 1 \pmod 7$
and $6 \mid (q-1)$, since procedure \textbf{\ref{totalsieve}($q,n$)} 
holds for all prime power $q<10^6$.

\section*{Appendix A: Procedures in SageMath}


{\scriptsize
\begin{procedure}
\KwIn{A real number $t>0$}
\KwOut{The constant $A_t$}
$p \gets 2$\\
$A \gets 1$\\
\While{$p<2^t$}
{
	$A \gets A \cdot \frac{2}{\sqrt[t]{p}}$\\
	$p \gets $ the next prime after $p$
}
\Return $A$
\caption{Constant($t$)}\label{At}
\end{procedure}
}

{\scriptsize
	\begin{procedure}
		\KwIn{A prime power $q$ and a positive integer $n$}
		\KwOut{a pair $(w_1,w_2)$ or $\emptyset$}
		$w \gets$ number of monic irreducible factors of $x^n-1$ in $\mathbb{F}_q[x]$\\
		\uIf{$\gcd(q,n)>1$}
			{$Pair \gets (w,w)$}
		\uElseIf{$\gcd(q-1,n)>1$}
			{$Pair \gets (w-1,w-2)$}
		\uElseIf{$\gcd(q+1,n)>1$}
			{$Pair \gets (w-1,w-1)$}
		\Else{$Pair \gets \emptyset$}
        \Return $Pair$
		\caption{NumberPolFactors($q,n$)}\label{valw1w2}
	\end{procedure}
}

{\scriptsize
\begin{procedure}
\KwIn{A prime power $q$, a positive integer $n$ and a real number $t>0$}
\KwOut{True or False}
$Pair \gets$\ref{valw1w2}($q,n$)\\
\uIf{$Pair \neq \emptyset$}
    {
    $(w_1,w_2)\gets Pair$\\
    $A\gets \mathrm{Constant}(t)$\\
	$res\gets \left[ q^{n/2-3}\ge 6^{2-1/t}\cdot A^2 \cdot q^{2\cdot n/t}\cdot 2^{w_1+w_2}\right]$
    }
\Else{$res \gets$ False}
\Return $res$
\caption{TestTheorem($q,n,t$)}\label{testAt}
	\end{procedure}
}

{\scriptsize
\begin{procedure}
\KwIn{A positive integer $T$ and a prime number $p_0$}
\KwOut{A pair $(S,u_0)$ where $S$ is a positive real number and $u_0$ is a positive integer}
$p \gets p_0$\\
$p_1 \gets p$\\
$S \gets 0$\\
$u_0 \gets 0$\\
\While{$T\ge p$ and $p<1000$}
	{
	\uIf{$p$ divides $T$}
		{
		$T \gets \frac{T}{p}$\\
		\If{$p=p_1$}
			{
			$S\gets S+ \frac{1}{p}$\\
			$u_0 \gets u_0+1$
			}
		$p_1 \gets $ next prime after $p$
		}
	\Else{
		$p \gets $ next prime after $p$\\
		$p_1 \gets p$
		}
	}
$p \gets p_1$\\
\While{$p<T$}
	{
	$S\gets S+ \frac{1}{p}$\\
	$u_0 \gets u_0+1$\\
	$T \gets T/p$\\
	$p \gets $ next prime after $p$
	}
\Return $(S,u_0)$
\caption{SumFactors($T,p_0$)}\label{sumfactors}
	\end{procedure}
}

{\scriptsize
\begin{procedure}
\KwIn{A prime power $q$, a positive integer $n$ and a prime number $p_0$}
\KwOut{True or False}
$Pair \gets$ \ref{valw1w2}($q,n$)\\
\uIf{$Pair \neq \emptyset$}
	{
	$(w_1,w_2)\gets Pair$\\
	$T \gets q^n-1$\\
	$p \gets 2$\\
	$\ell \gets 1$\\
	\While{$p<p_0$}
		{
		\uIf{$p \mid T$}
			{
			$T \gets \frac{T}{p}$\\
			$\ell \gets \ell \cdot p$
			}
		\Else{$p \gets $ next prime after $p$}
		}
	$p \gets p_0$\\
	$w\ell_1 \gets$ number of prime divisors of $\frac{\ell}{2}$\\ 
	$w\ell_2 \gets$ number of prime divisors of $\frac{\ell}{3}$\\
	$(S,u_0) \gets $ \ref{sumfactors}$(T,p_0)$\\
	$\delta \gets 1 - 2 S$\\
	\uIf{$\delta>0$}
		{
		$\Delta \gets 2+\frac{2\cdot u_0-1}{\delta}$  \nllabel{Delta}\\
		$res\gets \left[ q^{n/2-3}\ge 6^2\cdot \Delta \cdot 2^{w_1+w_2+w\ell_1+w\ell_2}\right]$
		}
	\Else{$res \gets$ False}
	}
\Else{$res \gets$ False}
\Return $res$
\caption{SpecialSieve($q,n,p_0$)}\label{specialsieve}
\end{procedure}
}

{\scriptsize
\IncMargin{1em}
\begin{procedure}
\KwIn{A prime power $q$ and a positive integer $n$}
\KwOut{a pair $(G_1,G_2)$ or $\emptyset$}
$\{g_1,\ldots ,g_s\} \gets$ list of monic irreducible factors of $x^n-1$ in $\mathbb{F}_q[x]$ ordered by degree\\
\uIf{$\gcd(q,n)>1$}
	{
	$G_1 \gets \{g_1,\ldots ,g_s\}$\\
	$G_2 \gets \{g_1,\ldots ,g_s\}$}
\uElseIf{$\gcd(q-1,n)>1$}
	{
	$G_1 \gets$ list of monic irreducible factors of $\frac{x^n-1}{g_1\cdot g_2}$
	in $\mathbb{F}_q[x]$ ordered by degree\\
	$G_2 \gets$ list of monic irreducible factors of $\frac{x^n-1}{g_1}$
	in $\mathbb{F}_q[x]$ ordered by degree\\
	}
\uElseIf{$\gcd(q+1,n)>1$}
	{
	$G_1 \gets$ list of monic irreducible factors of $\frac{x^n-1}{g_2}$
	in $\mathbb{F}_q[x]$ ordered by degree\\
	$G_2 \gets$ list of monic irreducible factors of $\frac{x^n-1}{g_1}$
	in $\mathbb{F}_q[x]$ ordered by degree\\
	}
\Else{$Pair \gets \emptyset$}
\uIf{$Pair=0$}
	{\Return $(G_1,G_2)$}
\Else{\Return $Pair$}
\caption{MonicFactors($q,n$)}\label{monicfactors}
	\end{procedure}
}

{\scriptsize
\begin{procedure}
\KwIn{Non-negative integers $i_1,i_2,j_1,j_2$ and lists $L_1,L_2,G_1,G_2$}
\KwOut{Lists $B_1,C_1,B_2,C_2,H_1,K_1,H_2,K_2$}
$Cond \gets$ True\\
$B_1\gets$ first $i_1$ elements of $L_1$\\
$C_1\gets$ last $\mathrm{len}(L_1)-i_1$ elements of $L_1$\\
$B_2\gets$ first $i_2$ elements of $L_2$\\
$C_2\gets$ last $\mathrm{len}(L_2)-i_2$ elements of $L_2$\\
$H_1\gets$ first $j_1$ elements of $G_1$\\
$K_1\gets$ last $\mathrm{len}(G_1)-j_1$ elements of $G_1$\\
$H_2\gets$ first $j_2$ elements of $G_2$\\
$K_2\gets$ last $\mathrm{len}(G_2)-j_2$ elements of $G_2$\\
\Return $(B_1,C_1,B_2,C_2,H_1,K_1,H_2,K_2)$
\caption{ListSieve($i_1,i_2,j_1,j_2,L_1,L_2,G_1,G_2$)}\label{listsieve}
	\end{procedure}
}

{\scriptsize
\begin{procedure}
\KwIn{A prime power $q$ and a positive integer $n$}
\KwOut{True or False}
$Cond \gets$ True\\
$L_1 \gets$ ordered list of prime divisors of $\frac{q^n-1}{2}$\\
$L_2 \gets$ ordered list of prime divisors of $\frac{q^n-1}{3}$\\
$Pair \gets$ \ref{monicfactors}($q,n$)\\
\uIf{$Pair \neq \emptyset$}
	{
	$(G_1,G_2) \gets Pair$\\
	$i_1,i_2,j_1,j_2 \gets 0$\\
	\While{$i_1\le \mathrm{len}(L_1)$ and $Cond$}
		{
		$(B_1,C_1,B_2,C_2,H_1,K_1,H_2,K_2)\gets$ \ref{listsieve}($i_1,i_2,j_1,j_2,L_1,L_2,G_1,G_2$)\\
		$\delta \gets 1 - \displaystyle
							\left(\sum_{p \in C_1} \frac{1}{p}+
		                        \sum_{p \in C_2} \frac{1}{p}+
		                        \sum_{g \in K_1} \frac{1}{q^{\deg g}} +
		                        \sum_{g \in K_2} \frac{1}{q^{\deg g}} \right)$ \\
		\uIf{$\delta>0$}
		    {
		    $\Delta\gets 2 + \dfrac{\mathrm{len}(C_1)+\mathrm{len}(C_2)+
		    	                    \mathrm{len}(K_1)+\mathrm{len}(K_2)-1}
	    	                   		{\delta}$\\
			$res \gets \left[
						q^{n/2-3}\ge 6 \cdot \Delta \cdot
												2^{\mathrm{len}(B_1)+\mathrm{len}(B_2)+
												\mathrm{len}(H_1)+\mathrm{len}(H_2)}
						\right]$\\
			$Cond \gets $ not $res$
			}
		\Else{$res \gets False$}
		$j_2\gets j_2+1$\\
		\If{$j_2>\mathrm{len}(G_2)$}
			{
			$j_2\gets 0$\\
			$j_1\gets j_1+1$\\
			\If{$j_1>\mathrm{len}(G_1)$}
				{
				$j_1\gets 0$\\
				$i_2\gets i_2+1$\\
				\If{$i_2>\mathrm{len}(L_2)$}
					{
					$i_2\gets 0$\\
					$i_1\gets i_1+1$
					}
				}
			}
		}
	}
\Else{$res \gets$ False}
\Return $res$
\caption{TotalSieve($q,n$)}\label{totalsieve}
	\end{procedure}
}

{\scriptsize
\begin{procedure}
\KwIn{Maximal and minimal bounds $q_{max},q_{min}$, a positive integer $n$ and a prime number $p_0$}
\KwOut{A pair $(q_{new},B)$ where $q_{new}$ is a new bound and $B$ is boolean}
$(p,a,n_0) \gets$ given by $n=p^a n_0$ where $p$ is the greatest prime divisor of $n$ and $\gcd(n,n_0)=1$\\
\uIf{$p$ is odd}{$\bar{p}\gets 2p^a$}
\Else{$\bar{p}\gets 2^a$}
$(e_1,e_2) \gets$ given as in the proof of Lemma \ref{proc-boundsieve} \\
$m_{max} \gets$ number of prime numbers less than $p_0$ which are not of the form $\bar{p}j+1$\\
$B \gets$ true\\
$Bound \gets \emptyset$\\
\For{$m \in \{ 2,\ldots, m_{max}\}$}
{
	$\mathcal{P}_0 \gets$ product of the first $m$ prime numbers which are not of the form $\bar{p}j+1$\\
	$(u_1,p_1,\mathcal{P}_{1},\mathcal{S}_{1}) \gets (0,p_0-1,1,0)$\\
	\While{$e_1 \cdot \mathcal{P}_0\cdot \mathcal{P}_{1}\le q_{max}^{\frac{n}{p}}-1$}
	{
		$(u_2,p_2,\mathcal{P}_{2},\mathcal{S}_{2}) \gets (0,2,1,0)$\\
		\While{$e_2 \cdot \mathcal{P}_0\cdot \mathcal{P}_{1}\cdot \mathcal{P}_{2} \le q_{max}^{n}-1$}
		{
			$\delta \gets 1 - 2(\mathcal{S}_{1}+ \mathcal{S}_{2}) - \frac{2n-3}{q_{min}}$\\
			\uIf{$\delta\le 0$}{$B\gets$ false}
				\Else{$\Delta \gets 2 + \frac{2(u_1+u_2)+2n-4}{\delta}$\\
						$Val \gets (6^2\cdot \Delta \cdot 2^{2m})^{\frac{2}{n-6}}$\\
						add $Val$ to the set $Bound$}
					$p_2\gets$ next prime after $p_2$ which is of the form $\bar{p}j+1$\\
					$(u_2,\mathcal{P}_2,\mathcal{S}_2) \gets (u_2+1,\mathcal{P}_2\cdot p_2,\mathcal{S}_2+\frac{1}{p_2})$
				}
				$p_1\gets$ next prime after $p_1$ which is not of the form $\bar{p}j+1$\\
				$(u_1,\mathcal{P}_1,\mathcal{S}_1) \gets (u_1+1,\mathcal{P}_1\cdot p_1,\mathcal{S}_1+\frac{1}{p_1})$ 
			}
		}
		$q_{new} \gets$ maximum value of the set $Bound$\\
		\Return $(q_{new},B)$
		\caption{BoundSieve($q_{\textrm{min}},q_{\textrm{max}},n,p_0$)}\label{boundsieve}
	\end{procedure}
}


\begin{thebibliography}{99}

\bibitem{AN}
J.J.R. Aguirre and V.G.L. Neumann,
{\em Existence of primitive $2$-normal elements in finite fields},
Finite Fields Appl. 73 (2021), 101864.
170--183.

\bibitem{AN2}
J.J.R. Aguirre, C. Carvalho and V.G.L. Neumann, {\em About r-primitive and k-normal elements in finite fields}, 
Des. Codes Cryptogr. (2022).

\bibitem{anju} Anju and R.K. Sharma, {\em Existence of some special primitive
normal elements over finite fields}, Finite Fields Appl. 46 (2017), 280-303.

\bibitem{carlitz} L. Carlitz, {\em Primitive roots in a finite field},
Trans. Amer. Math. Soc. 73 (1952), 373--382.

\bibitem{blum} 
M. Blum, S. Micali, {\em How to generate cryptographically strong
sequences of pseudorandom bits}, SIAM J. Comput. 13 (1984), 850--864. 



\bibitem{cgnt2} C.\ Carvalho, J.P.\ Guardieiro, V.G.L.\ Neumann and G.\ 
Tizziotti, {\em On the existence of pairs of primitive and normal
elements over finite fields},
Bull. Braz. Math. Soc. (N.S.) 53 (3) (2022), 677--699.

%
%

\bibitem{Cohen2021} S.D. Cohen and G. Kapetanakis, {\em Finite field extensions with the line or translate property for
$r$-primitive elements},  J. Aust. Math. Soc. 111 (3) (2021), 313--319.


\bibitem{Cohen2022} S.D. Cohen, G. Kapetanakis and L. Reis,
{\em The existence of $\mathbb{F}_q$-primitive points on curves using freeness},
preprint {\tt arXiv:2108.07373 [math.NT]} (2022).



\bibitem{CH2} S. D. Cohen and S. Huczynska, {\em The strong primitive normal 
basis theorem}. Acta Arith. 143 (4) (2010), 299--332.



\bibitem{davenport1} H. Davenport, {\em Bases for finite fields},
J. Lond. Math. Soc. 43 (1968), 21--39.

\bibitem{gao} S. Gao, {\em Elements of provable high orders in finite fields},
Proc. Amer. Math. Soc. 127 (1999), 1615--1623.

\bibitem{galois} R. Hachenberger and D. Jungnickel, {\em Topics in Galois Fields}, Springer, 2020.

\bibitem{haz} H. Hazarika, D.K. Basnet and S.D. Cohen,
{\em The existence of primitive normal elements of quadratic forms over finite fields},
J. Algebra Appl. 21 (4) (2022), 2250068.


\bibitem{haz2} H. Hazarika and D.K. Basnet, 
{\em On existence of primitive normal elements of rational form over finite fields of even characteristic}, 
Internat. J. Algebra Comput. 32 (2) (2022), 357--382.


\bibitem{Fu} L. Fu and D.Q. Wan, {\em A class of incomplete character sums},
Q. J. Math. 65 (2014), 1195--1211.

\bibitem{knormal} S. Huczynska, G.L. Mullen, D. Panario and D. Thomson,
{\em Existence and
properties of k-normal elements over finite fields}, Finite Fields Appl. 24 (2013),
170--183.

\bibitem{kap} G. Kapetanakis, {\em Normal bases and primitive elements over finite fields} Finite Fields Appl. 26 (2014), 
123--143.

\bibitem{ABS} A. \ Lemos, V.G.L.\ Neumann and S. Ribas, {\em On Arithmetic Progression of Primitive Elements in Finite Fields with one Normal}, preprint {\tt arXiv:2208.02876 [math.NT]} (2022).

\bibitem{lenstra} 
H.W. Lenstra and R. Schoof, {\em Primitive normal bases for finite fields},
Math. Comp. 48 (1987), 217--231.



\bibitem{LN} R. Lidl and H. Niederreiter, {\em Finite Fields}, Cambridge university press, 1997.

\bibitem{mel} G. Meletiou and G. Mullen, {\em A note on discrete logarithms in finite fields}, Appl. Algebra Engrg. Comm. Comput. 3 (1) (1992), 75--78.

\bibitem{negre}
C. Negre, {\em Finite field arithmetic using quasi-normal bases}, Finite Fields Appl. 13 (2007), 635--647.

\bibitem{RSS}
M. Rani, A. K. Sharma, S. K. Tiwari, {\em On r-primitive k-normal elements over finite fields}, Finite Fields Appl. 82 (2022), 102053.

\bibitem{rania} M. Rani, A. K. Sharma, S. K. Tiwari, A .Panigrahi,
{\em Inverses of r-primitive k-normal elements over finite fields}, arXiv:2201.11334 [math.NT].


\bibitem{lucas} L. Reis,
{\em Existence results on $k$-normal
elements over finite fields},
Rev. Mat. Iberoam. 35(3) (2019),
805--822.

\bibitem{lucas1}
L. Reis and D. Thompson, {\em Existence of primitive $1$-normal elements in finite fields}, Finite Fields Appl. 51 (2018), 238--269.





\bibitem{SAGE} The Sage Developers,
SageMath, the Sage Mathematics Software System (Version 8.1),
\texttt{https://www.sagemath.org}, 2020.


\bibitem{sozaya} J.A. Sozaya-Chan and H. Tapia-Recillas, On k-normal elements over finite fields, Finite Fields Appl.
52 (2018), 94--107.

\bibitem{zhang} A. Zhang and K. Feng, A New Criterion on k-Normal Elements over Finite Fields, Chin. Ann.Math. Ser. B 41 (2020), 665--678.

\end{thebibliography}
\end{document}